\documentclass[12pt,reqno]{amsart}
\usepackage{latexsym}
\usepackage{amssymb}
\usepackage{amscd}
\usepackage{comment}

\newif\iflang
\langtrue

\numberwithin{equation}{section}

\newtheorem{theorem}{Theorem}

\newtheorem{lemma}[theorem]{Lemma}
\newtheorem{corollary}[theorem]{Corollary}

\theoremstyle{remark}
\newtheorem*{remark}{Remark}

\iflang
\excludecomment{commlang}

\else
\excludecomment{commkurz}

\fi

\begin{document}
\title{Generalised Ap\'ery numbers modulo $9$}
\author[C. Krattenthaler and 
T.\,W. M\"uller]{C. Krattenthaler$^{\dagger}$ and
T. W. M\"uller$^*$} 

\address{$^{\dagger*}$Fakult\"at f\"ur Mathematik, Universit\"at Wien,
Oskar-Morgenstern-Platz~1, A-1090 Vienna, Austria.
WWW: {\tt http://www.mat.univie.ac.at/\lower0.5ex\hbox{\~{}}kratt}.}

\address{$^*$School of Mathematical Sciences, Queen Mary
\& Westfield College, University of London,
Mile End Road, London E1 4NS, United Kingdom.
}

\thanks{$^\dagger$Research partially supported by the Austrian
Science Foundation FWF, grants Z130-N13 and F50-N15,
the latter in the framework of the Special Research Program
``Algorithmic and Enumerative Combinatorics"
}

\subjclass[2010]{Primary 05A15;
Secondary 11A07 
}

\keywords{Ap\'ery numbers, congruences}

\begin{abstract}
We characterise the modular behaviour of (generalised) Ap\'ery number
modulo~$9$, thereby in particular establishing two conjectures in
{\it ``A 
method for determining the mod-$3^k$ behaviour of
recursive sequences"} [{\tt ar$\chi$iv:1308.2856}].
\end{abstract}
\maketitle

\section{Introduction}

For non-negative integers $r,s$ and $n$,
the (generalised) Ap\'ery numbers $a_n(r,s)$ are defined by
\begin{equation} \label{eq:Adef}
a_n(r,s)=\sum_{k=0}^n {\binom nk}^r {\binom {n+k}k}^s .
\end{equation}
The (classical) Ap\'ery numbers $a_n(2,1)$ and $a_n(2,2)$ appear
in Ap\'ery's proof \cite{ap} of the irrationality of $\zeta(2)$ and
$\zeta(3)$ (cf.\ also \cite{fi}) as the leading coefficients in
certain linear forms in $\zeta(2)$ and $1$, and in
$\zeta(3)$ and $1$, respectively. Similarly, the Ap\'ery numbers 
$a_n(r,1)$ appear as leading coefficients in linear forms in
$\zeta(r)$, $\zeta(r-2)$, \dots, $1$ in \cite{br}, where it is shown
that infinitely many values among $\zeta(3)$, $\zeta(5)$, $\zeta(7)$,
\dots\ are irrational. 
In the past,
the general Ap\'ery numbers $a_n(r,s)$ have
been the object of numerous arithmetic investigations,
see, for instance, \cite{CostAA,DeSaAA,GessAP,OnoKAA} for a
highly non-exhaustive selection. 

We point out that the Ap\'ery numbers with second parameter $s=0$
have received special attention in the literature. It is simple to see 
that $a_n(1,0)=2^n$ and $a_n(2,0)=\binom {2n}n$. 
The numbers $a_n(3,0)$ are also known under the name of ``Franel numbers,"
and the more general numbers $a_n(r,0)$ for $r\ge4$ as ``extended" or
``generalised Franel numbers." 

Given a prime number $p$ and the $p$-adic expansion of $n$,
$n=n_0+n_1p+n_2p^2+\dots+n_mp^m$, Deutsch and Sagan 
\cite[Theorem~5.9]{DeSaAA} show the factorisation
\begin{equation} \label{eq:A-p} 
a_n(r,s)\equiv
\prod _{i=0} ^{m}a_{n_i}(r,s)\pmod p,
\end{equation}
and use it to characterise the congruence classes of $a_n(r,s)$ modulo~$3$
in terms of precise conditions which the $3$-adic expansion of $n$ must
satisfy.
Gessel \cite{GessAP} obtains the above factorisation in the special case
where $r=s=2$, and furthermore, in Theorem~3(iii) of that paper, proves that
the factorisation \eqref{eq:A-p} remains valid for $r=s=2$ 
with $9$ in place of the modulus $p$.

The goal of our paper is to determine the behaviour of the Ap\'ery
numbers $a_n(r,s)$ modulo~$9$ for arbitrary positive integers $r$
and non-negative integers $s$. This will be
achieved by first deriving a Lucas-type result for binomial
coefficients modulo~$9$ (see Lemma~\ref{lem:Lucas9} in the next section), 
and using this to establish
an analogue of formula \eqref{eq:A-p} for the modulus~$9$;
see Theorems~\ref{thm:Afakt} and \ref{thm:Afakt0} in Section~\ref{sec:main}. 
From these theorems, one easily derives
explicit congruences for $a_n(r,s)$ depending on the congruence
classes of $r$ and $s$ modulo~$6$; 
\begin{commkurz}
\unskip
see Section~\ref{sec:expl}.
\end{commkurz}
\begin{commlang}
\unskip
see Section~\ref{sec:expl} for a sample of such results. 
In the long version \cite{KrMuZZ} of this paper, all possible cases
are worked out.
\end{commlang}
\unskip
In particular, Corollaries~\ref{thm:Ap27} and \ref{thm:Ap22} confirm 
Conjectures~65 and 66 from \cite{KrMuAE} 
concerning the explicit description
of the classical Ap\'ery numbers $a_n(2,1)$ and $a_n(2,2)$ modulo~$9$.
As a side result, we obtain generalisations of Gessel's mod-$9$ 
factorisation result for $a_n(2,2)$ mentioned above; 
see the remark after \eqref{eq:fmod3C}.

We point out that Rowland and Yassawi \cite{RoYaAA} have
also provided proofs for Conjectures~65 and 66 from \cite{KrMuAE}
(among many other things),
however using a completely different approach, based on
extracting diagonals from rational power series and 
construction of automata.
It is conceivable that their approach would also achieve proofs
of the other results in this paper.

It should be clear that our approach could also be
used to obtain explicit descriptions of the congruence classes of Ap\'ery
numbers modulo higher powers of~$3$,
and the same applies to the approach from \cite{RoYaAA}. 
The analysis and the results would
be more complex than the ones in this paper, though.

\section{A Lucas-type theorem modulo $9$}

The classical result of Lucas in \cite[p.~230, Eq.~(137)]{LucaAA} says that, 
if $p$ is a prime, $n=n_0+n_1p+n_2p^2+\dots+n_mp^m$,
and $k=k_0+k_1p+k_2p^2+\dots+k_mp^m$, $0\le n_i,k_i\le p-1$, then
$$
\binom nk\equiv
\prod _{i=0} ^{m}\binom {n_i}{k_i}\pmod p.
$$
Using the generating function approach proposed 
in \cite[Sec.~6]{GranAA}, we find an analogue of this formula for the
modulus~$9$, which is different from the mod-$9$ cases of the
generalisations of Lucas' formula to prime powers
given in \cite{DaWeAA} and \cite{GranAA}.

\begin{lemma} \label{lem:Lucas9}
For all non-negative integers $n=n_0+3n_1+9n_2+\dots+3^mn_m$
and $k=k_0+3k_1+9k_2+\dots+3^mn_m$, where $0\le n_i,k_i\le 2$
for all $i$, we have
\begin{align} 
\notag
\binom nk&\equiv
\binom {n_0}{k_0}\binom {n_1}{k_1}\cdots \binom {n_m}{k_m}
\\
\notag
&\kern.5cm
+\sum_{\nu=1}^m 3n_\nu \chi(n_{\nu-1}=0)
\binom {n_0}{k_0}\cdots\binom {n_{\nu-2}}{k_{\nu-2}}
\binom {1}{k_{\nu-1}-1}
\binom {n_{\nu}-1}{k_{\nu}}
\binom {n_{\nu+1}}{k_{\nu+1}}\cdots \binom {n_m}{k_m}
\\
\notag
&\kern.5cm
+\sum_{\nu=1}^m 3n_\nu \chi(n_{\nu-1}=1,\,k_{\nu-1}=0)
\binom {n_0}{k_0}\cdots\binom {n_{\nu-2}}{k_{\nu-2}}
\binom {n_{\nu}}{k_{\nu}}
\cdots \binom {n_m}{k_m}
\\
\notag
&\kern.5cm
-\sum_{\nu=1}^m 3n_\nu \chi(n_{\nu-1}=1)
\binom {n_0}{k_0}\cdots\binom {n_{\nu-2}}{k_{\nu-2}}
\binom {2}{k_{\nu-1}}
\binom {n_{\nu}-1}{k_{\nu}}
\binom {n_{\nu+1}}{k_{\nu+1}}\cdots \binom {n_m}{k_m}
\\
&\kern.5cm
+\sum_{\nu=1}^m 3n_\nu \chi(n_{\nu-1}=2,\,k_{\nu-1}=1)
\binom {n_0}{k_0}\cdots\binom {n_{\nu-2}}{k_{\nu-2}}
\binom {n_{\nu}}{k_{\nu}}
\cdots \binom {n_m}{k_m}
\quad (\text {\em mod }9),
\label{eq:Lucas9}
\end{align}
where $\chi(\mathcal S)=1$ if $\mathcal S$ is
true and $\chi(\mathcal S)=0$ otherwise.
\end{lemma}

\begin{proof}
During this proof,
given polynomials $f(z)$ and $g(z)$ with integer coefficients,
we write 
$$f(z)=g(z)~\text {modulo}~9$$ 
to mean that the coefficients
of $z^i$ in $f(z)$ and $g(z)$ agree modulo~$9$ for all $i$.
We use an analogous notation for the modulus~$3$.

An easy induction shows that
$$
(1+z)^{3^\nu}=1+z^{3^{\nu}}+3z^{3^{\nu-1}}\left(1+z^{3^{\nu-1}}\right)
\quad \text{modulo }9.
$$
This implies the expansion
\begin{align*} 
(1+z)^n&=
(1+z)^{n_0}
\big((1+z)^{3}\big)^{n_1}
\big((1+z)^{9}\big)^{n_2}
\cdots
\big((1+z)^{3^m}\big)^{n_m}\\
&=(1+z)^{n_0}
\prod _{\nu=1} ^{m}
\left(\left(1+z^{3^\nu}\right)
+3z^{3^{\nu-1}}\left(1+z^{3^{\nu-1}}\right)\right)^{n_\nu}\\
&=\prod _{\nu=0} ^{m}
\left(1+z^{3^\nu}\right)^{n_\nu}
+\sum_{\nu=1}^m
3n_\nu(1+z)^{n_0}
\cdots
\left(1+z^{3^{\nu-2}}\right)^{n_{\nu-2}}\\
&\kern4cm\cdot
z^{3^{\nu-1}}
\left(1+z^{3^{\nu-1}}\right)^{n_{\nu-1}+1}
\left(1+z^{3^\nu}\right)^{n_\nu-1}
\cdots
\left(1+z^{3^m}\right)^{n_m}
\\
&\kern11.5cm
\quad \text{modulo }9.
\end{align*}
Now, in case $n_{\nu-1}\ge 1$, one applies the
formulae
\begin{align*} 
z^{3^{\nu-1}}
\left(1+z^{3^{\nu-1}}\right)^{2}
&=\left(1+z^{3^\nu}\right)-\left(1+z^{3^{\nu-1}}\right)^{2}
\quad \text{modulo }3\\
z^{3^{\nu-1}}
\left(1+z^{3^{\nu-1}}\right)^3
&=z^{3^{\nu-1}}\left(1+z^{3^\nu}\right)
\quad \text{modulo }3
\end{align*}
to the terms involving $3^{\nu-1}$ in the sum, depending on whether
$n_{\nu-1}+1=2$ or $n_{\nu-1}+1=3$. Finally, every binomial term
$(1+z^{3^r})^{n'_r}$ is expanded using the binomial theorem, and
subsequently the coefficient of 
$$
z^k=z^{k_0+3k_1+9k_2+\dots+3^mk_m}
$$
is read off in all the terms. This leads directly to \eqref{eq:Lucas9}.
\end{proof}

\section{The main theorems}
\label{sec:main}

Here we prove the actual main results of our paper, namely
Theorems~\ref{thm:Afakt} and \ref{thm:Afakt0} below, 
which provide a refinement of
\eqref{eq:A-p} in the case where $p=3$ to modulus~$9$. The
explicit congruences for the Ap\'ery numbers $a_n(r,s)$
given in the next section are then simple consequences.
We state the results for $s\ge1$ and $s=0$ separately
in order to keep expressions at a moderate size.

\begin{theorem} \label{thm:Afakt}
For all positive integers $r$ and $s$, and 
non-negative integers $n=n_0+3n_1+9n_2+\dots+3^mn_m$,
where $0\le n_i\le 2$ for all $i$, we have
\begin{equation} \label{eq:Afakt} 
a_n(r,s)\equiv 
  \prod _{i=0} ^{m}a_{n_i}(r,s)+
  3\sum _{\nu=1} ^m
\left(\vphantom{ \sum_{k_{\nu-1}=0}^2}\right.
\underset {i\ne \nu-1,\nu}{\prod _{i=0} ^m} a_{n_i}(r,s)
\left.\vphantom{ \sum_{k_{\nu-1}=0}^2}\right)
f(n_{\nu-1},n_{\nu};r,s)
\quad (\text{\em mod 9}),
\end{equation}
where
\begin{multline} \label{eq:fdef1}
f(n_{\nu-1},n_{\nu};r,s)=
  n_{\nu-1} (n_{\nu-1}+1) n_{\nu} \Big(s (n_{\nu}+1) + 
  (-1)^s \big(s (n_{\nu}-1)  + r n_{\nu}\big) \Big)\\
+ 
 \chi(r=1)  (n_{\nu-1}+2) (n_{\nu-1}+1) n_{\nu} \\+ 
 \chi(s=1)  ( (-1)^r-1) (n_{\nu-1}-1) n_{\nu-1} n_{\nu}^2 .
\end{multline}
\end{theorem}

\begin{proof}
By the definition \eqref{eq:Adef} of the Ap\'ery numbers and the
Chu--Vandermonde summation formula, we may write
$$
a_n(r,s)=\sum_{k=0}^n {\binom nk}^r {\binom {n+k}k}^s 
=\sum_{k=0}^n {\binom nk}^r 
\left(\sum_{\ell=0}^n{\binom {n}\ell}{\binom {k}\ell}\right)^s .
$$
Now we apply the Lucas-type congruence from Lemma~\ref{lem:Lucas9}
to all three binomial coefficients on the right-hand side of the above
formula. Using the trivial congruences
$$
(a+3b)^s\equiv a^s+3a^{s-1}b\pmod 9
$$
and
$$
(a+3b)(c+3d)\equiv ac+3bc+3ad \pmod9,
$$
we expand the resulting expression. In this way, we obtain a large number
of terms. Writing $\ell=\ell_0+3\ell_1+9\ell_2+\dots+3^m\ell_m$,
the leading term is
$$
\sum_{0\le k_1,\dots,k_m\le2}
\left(
\prod _{i=0} ^{m}{\binom {n_i}{k_i}}^r
\right)
\left(
\sum_{0\le \ell_1,\dots,\ell_m\le2}
\prod _{i=0} ^{m}{\binom {n_i}{\ell_i}}
{\binom {k_i}{\ell_i}}
\right)^s,
$$
while one obtains 12 more terms which are similar.
At this point, the summation over the $\ell_i$'s can be carried out easily,
in all of the arising terms. This leads to the congruence
\begin{align} 
\notag
a_n(r,s)&\equiv
\prod _{i=0} ^{m}a_{n_i}(r,s)\\
\notag
&+
  3r\sum_{\nu=1}^m
\left(\vphantom{ \sum_{k_{\nu-1}=0}^2}\right.
\underset{i\ne \nu-1,\nu}{\prod _{i=0} ^{m}}a_{n_i}(r,s)
\left.\vphantom{ \sum_{k_{\nu-1}=0}^2}\right)
\left(\sum_{k_{\nu-1}=0}^2  \binom 1{k_{\nu-1}-1}
    {\binom {n_{\nu-1}} {k_{\nu-1}}}^{r-1} 
    {\binom {n_{\nu-1}+k_{\nu-1}} {k_{\nu-1}}}^s
    \right)\\
\notag
&\kern2cm
\cdot
\left(\sum_{k_\nu=0}^2 {\binom {n_\nu-1} {k_\nu}}
   {\binom {n_\nu} {k_\nu}}^{r-1} 
   {\binom {n_\nu+k_\nu} {k_\nu}}^s
 \right)
  \chi(n_{\nu-1}=0)n_\nu+\cdots\\
\notag
&+
  3s\sum_{\nu=1}^m
\left(\vphantom{ \sum_{k_{\nu-1}=0}^2}\right.
\underset{i\ne \nu-1,\nu}{\prod _{i=0} ^{m}}a_{n_i}(r,s)
\left.\vphantom{ \sum_{k_{\nu-1}=0}^2}\right)
 \left( \sum_{k_{\nu-1}=0}^2 \binom {k_{\nu-1}+1} {k_{\nu-1}-1}
    {\binom {n_{\nu-1}} {k_{\nu-1}}}^r 
    {\binom {n_{\nu-1}+k_{\nu-1}} {k_{\nu-1}}}^{s-1}
 \right)\\
\notag
&\kern2cm
\cdot
 \left( \sum_{k_\nu=0}^2 \binom {n_\nu+k_\nu-1} {k_\nu}
   {\binom {n_\nu} {k_\nu}}^r {\binom {n_\nu+k_\nu} {k_\nu}}^{s-1}
 \right)
  \chi(n_{\nu-1}=0)n_\nu+\cdots\\
\notag
&+  3s\sum_{\nu=1}^m
\left(\vphantom{ \sum_{k_{\nu-1}=0}^2}\right.
\underset{i\ne \nu-1,\nu}{\prod _{i=0} ^{m}}a_{n_i}(r,s)
\left.\vphantom{ \sum_{k_{\nu-1}=0}^2}\right)
 \left(  \sum_{k_{\nu-1}=0}^0 \binom {n_{\nu-1}+1} {n_{\nu-1}-1}
   {\binom {n_{\nu-1}} {k_{\nu-1}}}^r 
   {\binom {n_{\nu-1}+k_{\nu-1}} {k_{\nu-1}}}^{s-1}
 \right)\\
&\kern2cm
\cdot
 \left(  \sum_{k_\nu=0}^2 k_\nu \binom {n_\nu+k_\nu-1} {n_\nu}
    {\binom {n_\nu} {k_\nu}}^r {\binom {n_\nu+k_\nu} {k_\nu}}^{s-1}
 \right)+\cdots
\quad \pmod 9,
\label{eq:grosz}
\end{align}
where each of the dots $\cdots$ represents three similar terms.
(All the sums above over $\nu$ come from the second line in \eqref{eq:Lucas9}.)
The sums over $k_{\nu-1}$ and $k_\nu$ are written out explicitly.
The resulting formula is then simplified using elementary
congruences modulo~$3$. Namely, to simplify powers, we use 
\begin{equation*} 
n^a(n+1)^b\equiv (-1)^{b-1}n(n+1)\pmod3,
\quad \text{for $a,b\ge1$} .
\end{equation*}
Furthermore, by
\begin{equation} \label{eq:n=a}
\chi(n\equiv a\text{ (mod 3)})
\equiv -(n-a+1)(n-a+2)\pmod3 ,
\end{equation}
we may write
\begin{equation*} 
n^a
\equiv \chi(a=0)+(1-\chi(a=0))(-n(n+1)-n(n-1)(-1)^a) \pmod3,
\quad \text{for $a\ge0$} .
\end{equation*}
Finally, higher powers of $n_{\nu-1}$ and $n_\nu$ are lowered by
use of the (Fermat) congruence
\begin{equation*} 
n^3\equiv n\pmod 3.
\end{equation*}
After these manipulations, one collects terms, to obtain
a congruence of the form
$$
a_n(r,s)\equiv
\prod _{i=0} ^{m}a_{n_i}(r,s)\\
+  3\sum_{\nu=1}^m
\left(\vphantom{ \sum_{k_{\nu-1}=0}^2}\right.
\underset{i\ne \nu-1,\nu}{\prod _{i=0} ^{m}}a_{n_i}(r,s)
\left.\vphantom{ \sum_{k_{\nu-1}=0}^2}\right)
f(n_{\nu-1},n_{\nu};r,s).
$$
Upon considerable simplification, one sees that the term
$f(n_{\nu-1},n_{\nu};r,s)$ can be written as the term in
parentheses on the right-hand side of \eqref{eq:Afakt}.
\end{proof}

\begin{remark}
By applying \eqref{eq:n=a} again several times (namely ``backwards"),
the term\break $f(n_{\nu-1},n_{\nu};r,s)$ can alternatively be rewritten as
\begin{align}
\notag
f(n_{\nu-1},n_{\nu};r,s)
&\equiv 
\chi( n_{\nu-1}=1)  \Big(s \chi( n_{\nu}=1) +
  (-1)^s s \chi(n_{\nu}=2)  \\
\notag
&\kern5cm
+(-1)^s r\big(\chi( n_{\nu}=0)-1 \big)\Big)\\
\notag
&\kern1.5cm
- 
 \chi(r=1)  \chi(n_{\nu-1}=0) n_{\nu} \\
&\kern.5cm
+ 
 \chi(s=1)  ( (-1)^r-1) \chi(n_{\nu-1}=2)  \big(\chi(n_{\nu}=0)-1\big) 
\quad \text{(mod 3)}.
\label{eq:fdef}
\end{align}
\end{remark}

Following the same approach, we may establish an analogous result
for the (generalised) Franel numbers; that is, for the case where $s=0$.

\begin{theorem} \label{thm:Afakt0}
For all positive integers $r$ and 
non-negative integers $n=n_0+3n_1+9n_2+\dots+3^mn_m$,
where $0\le n_i\le 2$ for all $i$, we have
\begin{equation} \label{eq:Afakt0} 
a_n(r,0)\equiv 
  \prod _{i=0} ^{m}a_{n_i}(r,s)+
  3\sum _{\nu=1} ^m
\left(\vphantom{ \sum_{k_{\nu-1}=0}^2}\right.
\underset {i\ne \nu-1,\nu}{\prod _{i=0} ^m} a_{n_i}(r,s)
\left.\vphantom{ \sum_{k_{\nu-1}=0}^2}\right)
f(n_{\nu-1},n_{\nu};r,0)
\quad (\text{\em mod 9}),
\end{equation}
where
\begin{multline} \label{eq:fdef0}
f(n_{\nu-1},n_{\nu};r,0)=
\chi(r=1)
\big(  n_\nu -\chi(n_\nu=2)  \big)
\big( \chi(n_{\nu-1}=2)-1\big)
\kern4cm\\
-
  \big(n_\nu+\chi(n_\nu=2)(-1)^r \big)
\big(  \chi(n_{\nu-1}=1)-
    (-1)^r \chi(n_{\nu-1}=2)\big).
\end{multline}
\end{theorem}

\section{Explicit description of the Ap\'ery numbers modulo $9$}
\label{sec:expl}

We are now going to exploit Theorems~\ref{thm:Afakt} and 
\ref{thm:Afakt0} to obtain
explicit congruences modulo~$9$ for the Ap\'ery numbers $a_n(r,s)$,
depending on the congruence classes of $r$ and $s$ modulo~$6$.
In view of \eqref{eq:Afakt} and \eqref{eq:Afakt0}, 
we have to analyse the congruence
behaviour modulo~$9$ of $a_n(r,s)$ for $n=0,1,2$, as well as
the behaviour modulo~$3$ of the terms $f(n_{\nu-1},n_{\nu};r,s)$
given by \eqref{eq:fdef1} (or \eqref{eq:fdef}) and \eqref{eq:fdef0}.

We begin with the Ap\'ery numbers for small indices. We have
\begin{align} 
\label{eq:A90} 
a_0(r,s)&=1,\\
a_1(r,s)&=1+2^s\equiv\begin{cases} 
2\pmod 9,&\text{if }s\equiv 0\text{ (mod 6)},\\
3\pmod 9,&\text{if }s\equiv 1\text{ (mod 6)},\\
5\pmod 9,&\text{if }s\equiv 2\text{ (mod 6)},\\
0\pmod 9,&\text{if }s\equiv 3\text{ (mod 6)},\\
8\pmod 9,&\text{if }s\equiv 4\text{ (mod 6)},\\
6\pmod 9,&\text{if }s\equiv 5\text{ (mod 6)},
\end{cases}
\label{eq:A91} 
\end{align}
and
\begin{align}
a_2(r,s)&=1+2^r3^s+6^s\equiv\begin{cases} 
3\pmod 9,&\text{if }r\equiv 0\text{ (mod 6) and }s=0,\\
4\pmod 9,&\text{if }r\equiv 1\text{ (mod 6) and }s=0,\\
6\pmod 9,&\text{if }r\equiv 2\text{ (mod 6) and }s=0,\\
1\pmod 9,&\text{if }r\equiv 3\text{ (mod 6) and }s=0,\\
0\pmod 9,&\text{if }r\equiv 4\text{ (mod 6) and }s=0,\\
7\pmod 9,&\text{if }r\equiv 5\text{ (mod 6) and }s=0,\\
1\pmod 9,&\text{if }r\equiv 0,2\text{ (mod 3) and }s=1,\\
4\pmod 9,&\text{if }r\equiv 1\text{ (mod 3) and }s=1,\\
1\pmod 9,&\text{if }s\ge2.
\end{cases}
\label{eq:A92} 
\end{align}

Distinguishing between the various cases which arise when $r$ and $s$
run through the congruence classes modulo~$6$, we obtain
\begin{align} 
\notag
f&(n_{\nu-1},n_{\nu};r,s)
\\& \equiv \begin{cases} 
0
\pmod 3,
&\kern-3cm \text{if $r\equiv0$ (mod 3) and }s\equiv0\text{ (mod 3)},\\
\chi(n_{\nu-1} = 1) n_\nu
\pmod 3,
&\kern-3cm \text{if $r\equiv0$ (mod 6) and }s\equiv1\text{ (mod 6)},\\
\chi(n_{\nu-1} = 1) \big(\chi(n_\nu = 0)-1\big)
\pmod 3,
&\kern-3cm \text{if $r\equiv0$ (mod 3) and }s\equiv2\text{ (mod 6)},\\
\chi(n_{\nu-1} = 1) \big(1-\chi(n_\nu = 0)\big)
\pmod 3,
&\kern-3cm \text{if $r\equiv0$ (mod 3) and }s\equiv4\text{ (mod 6)},\\
- \chi(n_{\nu-1} = 1) n_\nu
\pmod 3,
&\kern-3cm \text{if $r\equiv0$ (mod 3) and }s\equiv5\text{ (mod 6)},\\
  \chi(n_{\nu-1} = 1) \big(\chi(n_\nu = 0)-1\big)
 -  \chi(n_{\nu-1} = 0) n_\nu
\pmod 3,
\\&\kern-3cm \text{if $r=1$ and }s\equiv0\text{ (mod 6)},\\
- n_{\nu-1} \chi(n_\nu=2) -n_\nu 
\pmod 3,
&\kern-3cm \text{if }r=s=1,\\
\end{cases}
\label{eq:fmod3A}
\end{align}
and
\begin{align} 
\notag
f&(n_{\nu-1},n_{\nu};r,s)
\\& \equiv \begin{cases} 
\chi(n_\nu=1)n_{\nu-1}  - n_\nu 
-\big(1-\chi( n_{\nu-1}=0)\big)\big(1-\chi( n_\nu=0)\big)
\pmod 3,
\\&\kern-6.1cm \text{if $r=1$, }s\equiv1\text{ (mod 6)},\text{ and $s\ge7$},\\
\chi( n_{\nu-1}=1 )\big(1-\chi( n_\nu=0)
- \chi( n_{\nu-1}=0\big) n_\nu
\pmod 3,
\\&\kern-6.1cm \text{if $r=1$ and }s\equiv2\text{ (mod 6)},\\
 \chi(n_{\nu-1} = 1) \big(1 - \chi(n_\nu = 0)\big) -
   \chi(n_{\nu-1} = 0) n_\nu
\pmod 3,
\\&\kern-6.1cm \text{if $r=1$ and }s\equiv3\text{ (mod 6)},\\
-\chi( n_{\nu-1}=0) n_\nu
\pmod 3,
&\kern-6.1cm \text{if $r=1$ and }s\equiv4\text{ (mod 6)},\\
\big(\chi( n_{\nu-1}=2)-1\big)n_\nu
+\chi( n_{\nu-1}=1) \big(1-\chi(n_\nu=0)\big)
\pmod 3,
\\&\kern-6.1cm \text{if $r=1$ and }s\equiv5\text{ (mod 6)},\\
\chi(n_{\nu-1} = 1) \big(\chi(n_\nu = 0)-1\big)
\pmod 3,
\\&\kern-6.1cm \text{if $r\equiv1$ (mod 3), $r\ge4$, and }s\equiv0\text{ (mod 6)},\\
\chi( n_{\nu-1}=1) n_\nu
+ \big(1-\chi(n_\nu=0) n_{\nu-1}
\pmod 3,
\\&\kern-6.1cm \text{if $r\equiv1$ (mod 6\big), $r\ge7$, and }s=1,\\
-\chi( n_{\nu-1}=1)\chi(n_\nu=1)
\pmod 3,
&\kern-6.1cm \text{if $r,s\equiv1$ (mod 6) and $r,s\ge7$},\\
&\kern-6.1cm \text{and if $r\equiv4$ (mod 6) and }s\equiv1\text{ (mod 6)},\\
\chi( n_{\nu-1}=1)  \big(1-\chi(n_\nu=0)\big)
\pmod 3,\\
&\kern-6.1cm \text{if $r\equiv1$ (mod 3), $r\ge4$,}\\
&\kern-6.1cm \text{and }s\equiv2,3\text{ (mod 6)},\\
0
\pmod 3,
&\kern-6.1cm \text{if $r\equiv1$ (mod 3), $r\ge4$, and }s\equiv4\text{ (mod 6)},\\
- \chi(n_{\nu-1} = 1) \chi(n_\nu = 2)
\pmod 3,\\
&\kern-6.1cm \text{if $r\equiv1$ (mod 3), $r\ge4$, and }s\equiv5\text{ (mod 6)},\\
 \chi(n_{\nu-1} = 1) \big(1 - \chi(n_\nu = 0)\big)
\pmod 3,
\\&\kern-6.1cm \text{if $r\equiv2$ (mod 3) and }s\equiv0\text{ (mod 6)},\\
\chi(n_{\nu-1} = 1) \chi(n_\nu = 2)
\pmod 3,
&\kern-6.1cm \text{if $r\equiv2$ (mod 6) and }s\equiv1\text{ (mod 6)},\\
0
\pmod 3,
&\kern-6.1cm \text{if $r\equiv2$ (mod 3) and }s\equiv2\text{ (mod 6)},\\
\chi(n_{\nu-1} = 1) \big(\chi(n_\nu = 0)-1\big)
\pmod 3,
\\&\kern-6.1cm \text{if $r\equiv2$ (mod 3) and }s\equiv3\text{ (mod 6)},\\
\chi(n_{\nu-1}=1)  \big(\chi(n_\nu=0)-1\big)
\pmod 3,
\\&\kern-6.1cm \text{if $r\equiv2$ (mod 3) and }s\equiv4\text{ (mod 6)},\\
\chi(n_{\nu-1} = 1) \chi(n_\nu = 1) 
\pmod 3,
&\kern-6.1cm \text{if $r\equiv2$ (mod 3) and }s\equiv5\text{ (mod 6)},\\
\chi(n_{\nu-1} = 2) \big(\chi(n_\nu = 0) -1\big)
+ \chi(n_{\nu-1} = 1)n_\nu
\pmod 3,
\\&\kern-6.1cm \text{if $r\equiv3$ (mod 6) and }s=1,\\
\chi(n_{\nu-1}=1) n_\nu
\pmod 3,
&\kern-6.1cm \text{if $r\equiv3$ (mod 6), }s\equiv1\text{ (mod 6), and }s\ge7,\\
\chi( n_{\nu-1}=1) n_\nu 
 -\big(1-\chi( n_{\nu-1}=0)\big)
\big(1-\chi( n_\nu=0)\big)
\pmod 3,
\\&\kern-6.1cm \text{if $r\equiv5$ (mod 6) and }s=1,\\
\chi(n_{\nu-1} = 1) \chi(n_\nu = 2)
\pmod 3,
&\kern-6.1cm \text{if $r\equiv5$ (mod 6), }s\equiv1\text{ (mod 6), and }s\ge7,\\
\end{cases}
\label{eq:fmod3B}
\end{align}
and
\begin{align} 
\notag
f&(n_{\nu-1},n_{\nu};r,s)
\\& \equiv \begin{cases} 
\big( \chi(n_{\nu-1} = 2)-\chi(n_{\nu-1} = 1)\big) 
\big(\chi(n_\nu = 2) + n_\nu\big)
\pmod 3,\\
&\kern-4cm \text{if $r\equiv0$ (mod 6), }r\ge6,\text{ and }s=0,\\
\big(1 + \chi(n_{\nu-1} = 1)\big) 
\big(\chi(n_\nu = 2) - n_\nu\big)
\pmod 3,\\
&\kern-4cm \text{if $r=1$ and }s=0,\\
\big(\chi(n_{\nu-1} = 1) + \chi(n_{\nu-1} = 2)\big) 
\big(\chi(n_\nu = 2) - n_\nu\big)
\pmod 3,\\
&\kern-4cm \text{if $r\equiv1,3,5$ (mod 6), }r\ge3,\text{ and }s=0,\\
\big( \chi(n_{\nu-1} = 2)-\chi(n_{\nu-1} = 1)\big) 
\big(\chi(n_\nu = 2) + n_\nu\big)
\pmod 3,\\
&\kern-4cm \text{if $r\equiv2,4$ (mod 6) and }s=0.
\end{cases}
\label{eq:fmod3C}
\end{align}

\begin{remark}
By examining \eqref{eq:fmod3A}--\eqref{eq:fmod3C}, one sees that Gessel's
result \cite[Theorem~3(iii)]{GessAP}, namely that
\begin{equation} \label{eq:A-0} 
a_n(r,s)\equiv
\prod _{i=0} ^{m}a_{n_i}(r,s)\pmod 9
\end{equation}
for $r=s=2$, does not only hold in that case, but more generally for
$r\equiv2$~(mod~$3$) and $s\equiv2$~(mod~$6$), and also for
$r\equiv s\equiv0$~(mod~$3$), and for
$r\equiv1$~(mod~$3$), $r\ge4$, and $s\equiv4$~(mod~$6$).
\end{remark}

If we combine Theorems~\ref{thm:Afakt} and \ref{thm:Afakt0} with
\eqref{eq:A90}--\eqref{eq:fmod3C},
\begin{commkurz}
\unskip
we obtain the following explicit descriptions of the mod-$9$ behaviour
of $a_n(r,s)$.
\end{commkurz}
\begin{commlang}
\unskip
we obtain detailed results concerning the mod-$9$ behaviour of
$a_n(r,s)$. We confine ourselves here to providing six results
representative for the total of 27 listed in \cite{KrMuZZ}.
\end{commlang}

\begin{corollary} \label{thm:Ap00}
If $r$ and $s$ are positive integers with 
$r\equiv0$~{\em(mod~$3$)} and $s\equiv0$~{\em(mod~$6$)},
then the Ap\'ery numbers $a_n(r,s)$ obey the following congruences
modulo $9$:
\begin{enumerate} 
\item [(i)]
$a_n(r,s)\equiv 1$~{\em(mod~$9$)} if, and only if,
the $3$-adic expansion of $n$ contains $6k$ digits~$1$, for some $k$,
and otherwise only $0$'s and $2$'s;
\item [(ii)]
$a_n(r,s)\equiv 2$~{\em(mod~$9$)} if, and only if,
the $3$-adic expansion of $n$ contains $6k+1$ digits~$1$, for some $k$,
and otherwise only $0$'s and $2$'s;
\item [(iii)]
$a_n(r,s)\equiv 4$~{\em(mod~$9$)} if, and only if,
the $3$-adic expansion of $n$ contains $6k+2$ digits~$1$, for some $k$,
and otherwise only $0$'s and $2$'s;
\item [(iv)]
$a_n(r,s)\equiv 5$~{\em(mod~$9$)} if, and only if,
the $3$-adic expansion of $n$ contains $6k+5$ digits~$1$, for some $k$,
and otherwise only $0$'s and $2$'s;
\item [(v)]
$a_n(r,s)\equiv 7$~{\em(mod~$9$)} if, and only if,
the $3$-adic expansion of $n$ contains $6k+4$ digits~$1$, for some $k$,
and otherwise only $0$'s and $2$'s;
\item [(vi)]
$a_n(r,s)\equiv 8$~{\em(mod~$9$)} if, and only if,
the $3$-adic expansion of $n$ contains $6k+3$ digits~$1$, for some $k$,
and otherwise only $0$'s and $2$'s;
\item[(vii)]in the cases not covered by Items~{\em(i)}--{\em(vi),}
$a_n(r,s)$ is divisible by $9$;
in particular, 
$a_n(r,s)\not\equiv 3,6$~{\em(mod~$9$)} for all $n$.
\end{enumerate}
\end{corollary}

\begin{corollary} \label{thm:Ap01}
If $r$ and $s$ are positive integers with 
$r\equiv0$~{\em(mod~$6$)} and $s\equiv1$~{\em(mod~$6$)}, or with
$r\equiv3$~{\em(mod~$6$)}, $s\equiv1$~{\em(mod~$6$)}, and $s\ge7$,
then the Ap\'ery numbers $a_n(r,s)$ obey the following congruences
modulo $9$:
\begin{enumerate} 
\item [(i)]
$a_n(r,s)\equiv 1$~{\em(mod~$9$)} if, and only if,
the $3$-adic expansion of $n$ contains $0$'s and $2$'s only;
\item [(ii)]
$a_n(r,s)\equiv 3$~{\em(mod~$9$)} if, and only if,
the $3$-adic expansion of $n$ has exactly one occurrence of
the string $01$ {\em(}including an occurrence of a $1$ at the
beginning{\em)} or of the string $11$ --- but not both ---
and otherwise contains only $0$'s and $2$'s;
\item [(iii)]in the cases not covered by Items~{\em(i)}--{\em(ii),}
$a_n(r,s)$ is divisible by $9$;
in particular, 
$a_n(r,s)\not\equiv 2,4,5,6,7,8$~{\em(mod~$9$)} for all $n$.
\end{enumerate}
\end{corollary}

\begin{commkurz}
\begin{corollary} \label{thm:Ap02}
If $r$ and $s$ are positive integers with 
$r\equiv0$~{\em(mod~$3$)} and $s\equiv2$~{\em(mod~$6$)}, 
or if 
$r\equiv1$~{\em(mod~$3$)}, $r\ge4$, and $s\equiv2$~{\em(mod~$6$)}, 
then the Ap\'ery numbers $a_n(r,s)$ obey the following congruences
modulo $9$:
\begin{enumerate} 
\item [(i)]
$a_n(r,s)\equiv 1$~{\em(mod~$9$)} if, and only if,
the $3$-adic expansion of $n$ has $2d$ digits~$1$,
$o_1$ occurrences of the string $11$, 
$o_2$ occurrences of the string $21$,
and $d+o_1-o_2\equiv0$~{\em(mod~$3$)}; 
\item [(ii)]
$a_n(r,s)\equiv 2$~{\em(mod~$9$)} if, and only if,
the $3$-adic expansion of $n$ has $2d+1$ digits~$1$,
$o_1$ occurrences of the string $11$, 
$o_2$ occurrences of the string $21$,
and $d+o_1-o_2\equiv2$~{\em(mod~$3$)}; 
\item [(iii)]
$a_n(r,s)\equiv 4$~{\em(mod~$9$)} if, and only if,
the $3$-adic expansion of $n$ has $2d$ digits~$1$,
$o_1$ occurrences of the string $11$, 
$o_2$ occurrences of the string $21$,
and $d+o_1-o_2\equiv2$~{\em(mod~$3$)}; 
\item [(iv)]
$a_n(r,s)\equiv 5$~{\em(mod~$9$)} if, and only if,
the $3$-adic expansion of $n$ has $2d+1$ digits~$1$,
$o_1$ occurrences of the string $11$, 
$o_2$ occurrences of the string $21$,
and $d+o_1-o_2\equiv0$~{\em(mod~$3$)}; 
\item [(v)]
$a_n(r,s)\equiv 7$~{\em(mod~$9$)} if, and only if,
the $3$-adic expansion of $n$ has $2d$ digits~$1$,
$o_1$ occurrences of the string $11$, 
$o_2$ occurrences of the string $21$,
and $d+o_1-o_2\equiv1$~{\em(mod~$3$)}; 
\item [(vi)]
$a_n(r,s)\equiv 8$~{\em(mod~$9$)} if, and only if,
the $3$-adic expansion of $n$ has $2d+1$ digits~$1$,
$o_1$ occurrences of the string $11$, 
$o_2$ occurrences of the string $21$,
and $d+o_1-o_2\equiv1$~{\em(mod~$3$)}; 
\item [(vii)]in the cases not covered by Items~{\em(i)}--{\em(vi),}
$a_n(r,s)$ is divisible by $9$;
in particular, 
$a_n(r,s)\not\equiv 3,6$~{\em(mod~$9$)} for all $n$.
\end{enumerate}
\end{corollary}
\end{commkurz}

\begin{corollary} \label{thm:Ap03}
If $r$ and $s$ are positive integers with 
$r\equiv0$~{\em(mod~$3$)} and $s\equiv3$~{\em(mod~$6$)},
then the Ap\'ery numbers $a_n(r,s)$ obey the following congruences
modulo $9$:
\begin{enumerate} 
\item [(i)]
$a_n(r,s)\equiv 1$~{\em(mod~$9$)} if, and only if,
the $3$-adic expansion of $n$ contains $0$'s and $2$'s only;
\item[(ii)]in all other cases $a_n(r,s)$ is divisible by $9$;
in particular, $a_n(r,s)\not\equiv 2,3,4,5,6,7,\break 8$~{\em(mod~$9$)}
for all $n$.
\end{enumerate}
\end{corollary}

\begin{commkurz}
\begin{corollary} \label{thm:Ap04}
If $r$ and $s$ are positive integers with 
$r\equiv0$~{\em(mod~$3$)} and $s\equiv4$~{\em(mod~$6$)}, 
then the Ap\'ery numbers $a_n(r,s)$ obey the following congruences
modulo $9$:
\begin{enumerate} 
\item [(i)]
$a_n(r,s)\equiv 1$~{\em(mod~$9$)} if, and only if,
the $3$-adic expansion of $n$ has an even number of digits~$1$,
and the difference of the number of occurrences of the
string $11$ and the number of occurrences of the string $21$
is $\equiv0$~{\em(mod~$3$)};
\item [(ii)]
$a_n(r,s)\equiv 2$~{\em(mod~$9$)} if, and only if,
the $3$-adic expansion of $n$ has an odd number of digits~$1$,
and the difference of the number of occurrences of the
string $11$ and the number of occurrences of the string $21$
is $\equiv2$~{\em(mod~$3$)};
\item [(iii)]
$a_n(r,s)\equiv 4$~{\em(mod~$9$)} if, and only if,
the $3$-adic expansion of $n$ has an even number of digits~$1$,
and the difference of the number of occurrences of the
string $11$ and the number of occurrences of the string $21$
is $\equiv1$~{\em(mod~$3$)};
\item [(iv)]
$a_n(r,s)\equiv 5$~{\em(mod~$9$)} if, and only if,
the $3$-adic expansion of $n$ has an odd number of digits~$1$,
and the difference of the number of occurrences of the
string $11$ and the number of occurrences of the string $21$
is $\equiv1$~{\em(mod~$3$)};
\item [(v)]
$a_n(r,s)\equiv 7$~{\em(mod~$9$)} if, and only if,
the $3$-adic expansion of $n$ has an even number of digits~$1$,
and the difference of the number of occurrences of the
string $11$ and the number of occurrences of the string $21$
is $\equiv2$~{\em(mod~$3$)};
\item [(vi)]
$a_n(r,s)\equiv 8$~{\em(mod~$9$)} if, and only if,
the $3$-adic expansion of $n$ has an odd number of digits~$1$,
and the difference of the number of occurrences of the
string $11$ and the number of occurrences of the string $21$
is $\equiv0$~{\em(mod~$3$)};
\item [(vii)]in the cases not covered by Items~{\em(i)}--{\em(vi),}
$a_n(r,s)$ is divisible by $9$;
in particular, 
$a_n(r,s)\not\equiv 3,6$~{\em(mod~$9$)} for all $n$.
\end{enumerate}
\end{corollary}

\begin{corollary} \label{thm:Ap05}
If $r$ and $s$ are positive integers with 
$r\equiv0$~{\em(mod~$6$)} and $s\equiv5$~{\em(mod~$6$)}, 
then the Ap\'ery numbers $a_n(r,s)$ obey the following congruences
modulo $9$:
\begin{enumerate} 
\item [(i)]
$a_n(r,s)\equiv 1$~{\em(mod~$9$)} if, and only if,
the $3$-adic expansion of $n$ contains $0$'s and $2$'s only;
\item [(ii)]
$a_n(r,s)\equiv 6$~{\em(mod~$9$)} if, and only if,
the $3$-adic expansion of $n$ has exactly one occurrence of
the string $01$ {\em(}including an occurrence of a $1$ at the
beginning{\em)} or of the string $11$ --- but not both ---
and otherwise contains only $0$'s and $2$'s;
\item [(iii)]in the cases not covered by Items~{\em(i)}--{\em(ii),}
$a_n(r,s)$ is divisible by $9$;
in particular, 
$a_n(r,s)\not\equiv 2,3,4,5,7,8$~{\em(mod~$9$)} for all $n$.
\end{enumerate}
\end{corollary}

\begin{corollary} \label{thm:Ap16}
If $s$ is a positive integer with 
$s\equiv0$~{\em(mod~$6$)}, 
then the Ap\'ery numbers $a_n(1,s)$ obey the following congruences
modulo $9$:
\begin{enumerate} 
\item [(i)]
$a_n(1,s)\equiv 1$~{\em(mod~$9$)} if, and only if,
the $3$-adic expansion of $n$ has $2d$ digits~$1$,
$o_1$ occurrences of the string $11$, 
$o_2$ occurrences of the string $21$,
$o_3$ occurrences of the string $10$,
$o_4$ occurrences of the string $20$,
and $d-o_1+o_2+o_3+o_4\equiv0$~{\em(mod~$3$)}; 
\item [(ii)]
$a_n(1,s)\equiv 2$~{\em(mod~$9$)} if, and only if,
the $3$-adic expansion of $n$ has $2d+1$ digits~$1$,
$o_1$ occurrences of the string $11$, 
$o_2$ occurrences of the string $21$,
$o_3$ occurrences of the string $10$,
$o_4$ occurrences of the string $20$,
and $d-o_1+o_2+o_3+o_4\equiv0$~{\em(mod~$3$)}; 
\item [(iii)]
$a_n(1,s)\equiv 4$~{\em(mod~$9$)} if, and only if,
the $3$-adic expansion of $n$ has $2d$ digits~$1$,
$o_1$ occurrences of the string $11$, 
$o_2$ occurrences of the string $21$,
$o_3$ occurrences of the string $10$,
$o_4$ occurrences of the string $20$,
and $d-o_1+o_2+o_3+o_4\equiv1$~{\em(mod~$3$)}; 
\item [(iv)]
$a_n(1,s)\equiv 5$~{\em(mod~$9$)} if, and only if,
the $3$-adic expansion of $n$ has $2d+1$ digits~$1$,
$o_1$ occurrences of the string $11$, 
$o_2$ occurrences of the string $21$,
$o_3$ occurrences of the string $10$,
$o_4$ occurrences of the string $20$,
and $d-o_1+o_2+o_3+o_4\equiv2$~{\em(mod~$3$)}; 
\item [(v)]
$a_n(1,s)\equiv 7$~{\em(mod~$9$)} if, and only if,
the $3$-adic expansion of $n$ has $2d$ digits~$1$,
$o_1$ occurrences of the string $11$, 
$o_2$ occurrences of the string $21$,
$o_3$ occurrences of the string $10$,
$o_4$ occurrences of the string $20$,
and $d-o_1+o_2+o_3+o_4\equiv2$~{\em(mod~$3$)}; 
\item [(vi)]
$a_n(1,s)\equiv 8$~{\em(mod~$9$)} if, and only if,
the $3$-adic expansion of $n$ has $2d+1$ digits~$1$,
$o_1$ occurrences of the string $11$, 
$o_2$ occurrences of the string $21$,
$o_3$ occurrences of the string $10$,
$o_4$ occurrences of the string $20$,
and $d-o_1+o_2+o_3+o_4\equiv1$~{\em(mod~$3$)}; 
\item [(vii)]in the cases not covered by Items~{\em(i)}--{\em(vi),}
$a_n(1,s)$ is divisible by $9$;
in particular, 
$a_n(1,s)\not\equiv 3,6$~{\em(mod~$9$)} for all $n$.
\end{enumerate}
\end{corollary}

\begin{corollary} \label{thm:Ap11}
Let $n_0$ denote the $0$-th digit in the $3$-adic representation of 
the non-negative integer $n$.
The Ap\'ery numbers $a_n(1,1)$ obey the following congruences
modulo $9$:
\begin{enumerate} 
\item [(i)]
$a_n(1,1)\equiv 1$~{\em(mod~$9$)} if, and only if,
the $3$-adic expansion of $n$ contains $0$'s and $2$'s only,
and the number of maximal strings of $2$'s is
$\equiv n_0$~{\em(mod~$3$)};
\item [(ii)]
$a_n(1,1)\equiv 3$~{\em(mod~$9$)} if, and only if,
$n_0=1$;
\item [(iii)]
$a_n(1,1)\equiv 4$~{\em(mod~$9$)} if, and only if,
the $3$-adic expansion of $n$ contains $0$'s and $2$'s only,
and the number of maximal strings of $2$'s is
$\equiv n_0+2$~{\em(mod~$3$)};
\item [(iv)]
$a_n(1,1)\equiv 6$~{\em(mod~$9$)} if, and only if,
the $3$-adic expansion of $n$ has exactly one occurrence
of the string $11$
and otherwise contains only $0$'s and $2$'s;
\item [(v)]
$a_n(1,1)\equiv 7$~{\em(mod~$9$)} if, and only if,
the $3$-adic expansion of $n$ contains $0$'s and $2$'s only,
and the number of maximal strings of $2$'s is
$\equiv n_0+1$~{\em(mod~$3$)};
\item [(vi)]in the cases not covered by Items~{\em(i)}--{\em(v),}
$a_n(1,1)$ is divisible by $9$;
in particular, 
$a_n(1,1)\not\equiv 2,5,8$~{\em(mod~$9$)} for all $n$.
\end{enumerate}
\end{corollary}

\begin{corollary} \label{thm:Ap17}
If $s$ is a positive integer with 
$s\equiv1$~{\em(mod~$6$)} and $s\ge7$,
then the Ap\'ery numbers $a_n(1,s)$ obey the following congruences
modulo $9$:
\begin{enumerate} 
\item [(i)]
$a_n(1,s)\equiv 1$~{\em(mod~$9$)} if, and only if,
the $3$-adic expansion of $n$ contains $0$'s and $2$'s only,
and the number of strings $20$ is
$\equiv 0$~{\em(mod~$3$)};
\item [(ii)]
$a_n(1,s)\equiv 3$~{\em(mod~$9$)} if, and only if,
the $3$-adic expansion of $n$ has exactly one occurrence of $1$,
no occurrence of the string $10$,
and otherwise contains only $0$'s and $2$'s;
\item [(iii)]
$a_n(1,s)\equiv 4$~{\em(mod~$9$)} if, and only if,
the $3$-adic expansion of $n$ contains $0$'s and $2$'s only,
and the number of strings $20$ is
$\equiv 1$~{\em(mod~$3$)};
\item [(iv)]
$a_n(1,s)\equiv 6$~{\em(mod~$9$)} if, and only if,
the $3$-adic expansion of $n$ has exactly one occurrence
of the string $11$
and otherwise contains only $0$'s and $2$'s;
\item [(v)]
$a_n(1,s)\equiv 7$~{\em(mod~$9$)} if, and only if,
the $3$-adic expansion of $n$ contains $0$'s and $2$'s only,
and the number of strings $20$ is
$\equiv 2$~{\em(mod~$3$)};
\item [(vi)]in the cases not covered by Items~{\em(i)}--{\em(v),}
$a_n(1,s)$ is divisible by $9$;
in particular, 
$a_n(1,s)\not\equiv 2,5,8$~{\em(mod~$9$)} for all $n$.
\end{enumerate}
\end{corollary}

\begin{corollary} \label{thm:Ap12}
If $s$ is a positive integer with 
$s\equiv2$~{\em(mod~$6$)}, 
then the Ap\'ery numbers $a_n(1,s)$ obey the following congruences
modulo $9$:
\begin{enumerate} 
\item [(i)]
$a_n(1,s)\equiv 1$~{\em(mod~$9$)} if, and only if,
the $3$-adic expansion of $n$ has $2d$ digits~$1$,
$o_1$ occurrences of the string $11$, 
$o_2$ occurrences of the string $21$,
$o_3$ occurrences of the string $10$,
$o_4$ occurrences of the string $20$,
and $d-o_1+o_2-o_3-o_4\equiv0$~{\em(mod~$3$)}; 
\item [(ii)]
$a_n(1,s)\equiv 2$~{\em(mod~$9$)} if, and only if,
the $3$-adic expansion of $n$ has $2d+1$ digits~$1$,
$o_1$ occurrences of the string $11$, 
$o_2$ occurrences of the string $21$,
$o_3$ occurrences of the string $10$,
$o_4$ occurrences of the string $20$,
and $d-o_1+o_2-o_3-o_4\equiv2$~{\em(mod~$3$)}; 
\item [(iii)]
$a_n(1,s)\equiv 4$~{\em(mod~$9$)} if, and only if,
the $3$-adic expansion of $n$ has $2d$ digits~$1$,
$o_1$ occurrences of the string $11$, 
$o_2$ occurrences of the string $21$,
$o_3$ occurrences of the string $10$,
$o_4$ occurrences of the string $20$,
and $d-o_1+o_2-o_3-o_4\equiv2$~{\em(mod~$3$)}; 
\item [(iv)]
$a_n(1,s)\equiv 5$~{\em(mod~$9$)} if, and only if,
the $3$-adic expansion of $n$ has $2d+1$ digits~$1$,
$o_1$ occurrences of the string $11$, 
$o_2$ occurrences of the string $21$,
$o_3$ occurrences of the string $10$,
$o_4$ occurrences of the string $20$,
and $d-o_1+o_2-o_3-o_4\equiv0$~{\em(mod~$3$)}; 
\item [(v)]
$a_n(1,s)\equiv 7$~{\em(mod~$9$)} if, and only if,
the $3$-adic expansion of $n$ has $2d$ digits~$1$,
$o_1$ occurrences of the string $11$, 
$o_2$ occurrences of the string $21$,
$o_3$ occurrences of the string $10$,
$o_4$ occurrences of the string $20$,
and $d-o_1+o_2-o_3-o_4\equiv1$~{\em(mod~$3$)}; 
\item [(vi)]
$a_n(1,s)\equiv 8$~{\em(mod~$9$)} if, and only if,
the $3$-adic expansion of $n$ has $2d+1$ digits~$1$,
$o_1$ occurrences of the string $11$, 
$o_2$ occurrences of the string $21$,
$o_3$ occurrences of the string $10$,
$o_4$ occurrences of the string $20$,
and $d-o_1+o_2-o_3-o_4\equiv1$~{\em(mod~$3$)}; 
\item [(vii)]in the cases not covered by Items~{\em(i)}--{\em(vi),}
$a_n(1,s)$ is divisible by $9$;
in particular, 
$a_n(1,s)\not\equiv 3,6$~{\em(mod~$9$)} for all $n$.
\end{enumerate}
\end{corollary}

\begin{corollary} \label{thm:Ap13}
If $s$ is a positive integer with 
$s\equiv3$~{\em(mod~$6$)},
then the Ap\'ery numbers $a_n(1,s)$ obey the following congruences
modulo $9$:
\begin{enumerate} 
\item [(i)]
$a_n(1,s)\equiv 1$~{\em(mod~$9$)} if, and only if,
the $3$-adic expansion of $n$ contains $0$'s and $2$'s only,
and the number of occurrences of the string $20$ 
is $\equiv0$~{\em(mod~$3$)};
\item [(ii)]
$a_n(1,s)\equiv 3$~{\em(mod~$9$)} if, and only if,
the $3$-adic expansion of $n$ has exactly one occurrence of
the string $11$, or exactly one occurrence of the string $21$
but no occurrence of the string $11$,
and otherwise contains only $0$'s and $2$'s;
\item [(iii)]
$a_n(1,s)\equiv 4$~{\em(mod~$9$)} if, and only if,
the $3$-adic expansion of $n$ contains $0$'s and $2$'s only,
and the number of occurrences of the string $20$ 
is $\equiv1$~{\em(mod~$3$)};
\item [iv)]
$a_n(1,s)\equiv 6$~{\em(mod~$9$)} if, and only if,
the $3$-adic expansion of $n$ has exactly one occurrence of
the string $10$, no occurrences of the strings $11$ or $21$, 
and otherwise contains only $0$'s and $2$'s;
\item [(v)]
$a_n(1,s)\equiv 7$~{\em(mod~$9$)} if, and only if,
the $3$-adic expansion of $n$ contains $0$'s and $2$'s only,
and the number of occurrences of the string $20$ 
is $\equiv2$~{\em(mod~$3$)};
\item [(vi)]in the cases not covered by Items~{\em(i)}--{\em(v),}
$a_n(1,s)$ is divisible by $9$;
in particular, $a_n(1,s)\not\equiv 2,5,8$~{\em(mod~$9$)}
for all $n$.
\end{enumerate}
\end{corollary}

\begin{corollary} \label{thm:Ap14}
If $s$ is a positive integer with 
$s\equiv4$~{\em(mod~$6$)}, 
then the Ap\'ery numbers $a_n(1,s)$ obey the following congruences
modulo $9$:
\begin{enumerate} 
\item [(i)]
$a_n(1,s)\equiv 1$~{\em(mod~$9$)} if, and only if,
the $3$-adic expansion of $n$ has an even number of digits~$1$,
and the sum of the number of occurrences of the
string $10$ and the number of occurrences of the string $20$
is $\equiv0$~{\em(mod~$3$)};
\item [(ii)]
$a_n(1,s)\equiv 2$~{\em(mod~$9$)} if, and only if,
the $3$-adic expansion of $n$ has an odd number of digits~$1$,
and the sum of the number of occurrences of the
string $10$ and the number of occurrences of the string $20$
is $\equiv2$~{\em(mod~$3$)};
\item [(iii)]
$a_n(1,s)\equiv 4$~{\em(mod~$9$)} if, and only if,
the $3$-adic expansion of $n$ has an even number of digits~$1$,
and the sum of the number of occurrences of the
string $10$ and the number of occurrences of the string $20$
is $\equiv1$~{\em(mod~$3$)};
\item [(iv)]
$a_n(1,s)\equiv 5$~{\em(mod~$9$)} if, and only if,
the $3$-adic expansion of $n$ has an odd number of digits~$1$,
and the sum of the number of occurrences of the
string $10$ and the number of occurrences of the string $20$
is $\equiv1$~{\em(mod~$3$)};
\item [(v)]
$a_n(1,s)\equiv 7$~{\em(mod~$9$)} if, and only if,
the $3$-adic expansion of $n$ has an even number of digits~$1$,
and the sum of the number of occurrences of the
string $10$ and the number of occurrences of the string $20$
is $\equiv2$~{\em(mod~$3$)};
\item [(vi)]
$a_n(1,s)\equiv 8$~{\em(mod~$9$)} if, and only if,
the $3$-adic expansion of $n$ has an odd number of digits~$1$,
and the sum of the number of occurrences of the
string $10$ and the number of occurrences of the string $20$
is $\equiv0$~{\em(mod~$3$)};
\item [(vii)]in the cases not covered by Items~{\em(i)}--{\em(vi),}
$a_n(1,s)$ is divisible by $9$;
in particular, 
$a_n(1,s)\not\equiv 3,6$~{\em(mod~$9$)} for all $n$.
\end{enumerate}
\end{corollary}

\begin{corollary} \label{thm:Ap15}
If $s$ is a positive integer with 
$s\equiv5$~{\em(mod~$6$)}, 
then the Ap\'ery numbers $a_n(1,s)$ obey the following congruences
modulo $9$:
\begin{enumerate} 
\item [(i)]
$a_n(1,s)\equiv 1$~{\em(mod~$9$)} if, and only if,
the $3$-adic expansion of $n$ contains $0$'s and $2$'s only,
and the number of strings $20$ is
$\equiv 0$~{\em(mod~$3$)};
\item [(ii)]
$a_n(1,s)\equiv 3$~{\em(mod~$9$)} if, and only if,
the $3$-adic expansion of $n$ has exactly one occurrence of 
the string $10$ or of the string $21$ --- but not both,
and otherwise contains only $0$'s and $2$'s;
\item [(iii)]
$a_n(1,s)\equiv 4$~{\em(mod~$9$)} if, and only if,
the $3$-adic expansion of $n$ contains $0$'s and $2$'s only,
and the number of strings $20$ is
$\equiv 1$~{\em(mod~$3$)};
\item [(iv)]
$a_n(1,s)\equiv 6$~{\em(mod~$9$)} if, and only if,
the $3$-adic expansion of $n$ has exactly one occurrence
of $1$ but no occurrence of the strings $10$ or $21$,
and otherwise contains only $0$'s and $2$'s;
\item [(v)]
$a_n(1,s)\equiv 7$~{\em(mod~$9$)} if, and only if,
the $3$-adic expansion of $n$ contains $0$'s and $2$'s only,
and the number of strings $20$ is
$\equiv 2$~{\em(mod~$3$)};
\item [(vi)]in the cases not covered by Items~{\em(i)}--{\em(v),}
$a_n(1,s)$ is divisible by $9$;
in particular, 
$a_n(1,s)\not\equiv 2,5,8$~{\em(mod~$9$)} for all $n$.
\end{enumerate}
\end{corollary}

\begin{corollary} \label{thm:Ap76}
If $r$ and $s$ are positive integers with 
$r\equiv1$~{\em(mod~$3$)}, $r\ge4$, and $s\equiv0$~{\em(mod~$6$)}, 
then the Ap\'ery numbers $a_n(r,s)$ obey the following congruences
modulo $9$:
\begin{enumerate} 
\item [(i)]
$a_n(r,s)\equiv 1$~{\em(mod~$9$)} if, and only if,
the $3$-adic expansion of $n$ has $2d$ digits~$1$,
$o_1$ occurrences of the string $11$, 
$o_2$ occurrences of the string $21$,
and $d-o_1+o_2\equiv0$~{\em(mod~$3$)}; 
\item [(ii)]
$a_n(r,s)\equiv 2$~{\em(mod~$9$)} if, and only if,
the $3$-adic expansion of $n$ has $2d+1$ digits~$1$,
$o_1$ occurrences of the string $11$, 
$o_2$ occurrences of the string $21$,
and $d-o_1+o_2\equiv2$~{\em(mod~$3$)}; 
\item [(iii)]
$a_n(r,s)\equiv 4$~{\em(mod~$9$)} if, and only if,
the $3$-adic expansion of $n$ has $2d$ digits~$1$,
$o_1$ occurrences of the string $11$, 
$o_2$ occurrences of the string $21$,
and $d-o_1+o_2\equiv1$~{\em(mod~$3$)}; 
\item [(iv)]
$a_n(r,s)\equiv 5$~{\em(mod~$9$)} if, and only if,
the $3$-adic expansion of $n$ has $2d+1$ digits~$1$,
$o_1$ occurrences of the string $11$, 
$o_2$ occurrences of the string $21$,
and $d-o_1+o_2\equiv2$~{\em(mod~$3$)}; 
\item [(v)]
$a_n(r,s)\equiv 7$~{\em(mod~$9$)} if, and only if,
the $3$-adic expansion of $n$ has $2d$ digits~$1$,
$o_1$ occurrences of the string $11$, 
$o_2$ occurrences of the string $21$,
and $d-o_1+o_2\equiv2$~{\em(mod~$3$)}; 
\item [(vi)]
$a_n(r,s)\equiv 8$~{\em(mod~$9$)} if, and only if,
the $3$-adic expansion of $n$ has $2d+1$ digits~$1$,
$o_1$ occurrences of the string $11$, 
$o_2$ occurrences of the string $21$,
and $d-o_1+o_2\equiv1$~{\em(mod~$3$)}; 
\item [(vii)]in the cases not covered by Items~{\em(i)}--{\em(vi),}
$a_n(r,s)$ is divisible by $9$;
in particular, 
$a_n(r,s)\not\equiv 3,6$~{\em(mod~$9$)} for all $n$.
\end{enumerate}
\end{corollary}

\begin{corollary} \label{thm:Ap71}
If $r$ is a positive integer with 
$r\equiv1$~{\em(mod~$6$)} and $r\ge7$,
then the Ap\'ery numbers $a_n(r,1)$ obey the following congruences
modulo $9$:
\begin{enumerate} 
\item [(i)]
$a_n(r,1)\equiv 1$~{\em(mod~$9$)} if, and only if,
the $3$-adic expansion of $n$ contains $0$'s and $2$'s only,
and the number of maximal strings of $2$'s is
$\equiv 0$~{\em(mod~$3$)};
\item [(ii)]
$a_n(r,1)\equiv 3$~{\em(mod~$9$)} if, and only if,
the $3$-adic expansion of $n$ has exactly one occurrence of $1$
but no occurrence of the string $12$, 
and otherwise contains only $0$'s and $2$'s;
\item [(iii)]
$a_n(r,1)\equiv 4$~{\em(mod~$9$)} if, and only if,
the $3$-adic expansion of $n$ contains $0$'s and $2$'s only,
and the number of maximal strings of $2$'s is
$\equiv 1$~{\em(mod~$3$)};
\item [(iv)]
$a_n(r,1)\equiv 6$~{\em(mod~$9$)} if, and only if,
the $3$-adic expansion of $n$ has exactly one occurrence
of the string $11$
and otherwise contains only $0$'s and $2$'s;
\item [(v)]
$a_n(r,1)\equiv 7$~{\em(mod~$9$)} if, and only if,
the $3$-adic expansion of $n$ contains $0$'s and $2$'s only,
and the number of maximal strings of $2$'s is
$\equiv 2$~{\em(mod~$3$)};
\item [(vi)]in the cases not covered by Items~{\em(i)}--{\em(v),}
$a_n(r,1)$ is divisible by $9$;
in particular, 
$a_n(r,1)\not\equiv 2,5,8$~{\em(mod~$9$)} for all $n$.
\end{enumerate}
\end{corollary}

\begin{corollary} \label{thm:Ap77}
If $r$ and $s$ are positive integers with 
$r,s\equiv1$~{\em(mod~$6$)} and $r,s\ge7$, or with
$r\equiv4$~{\em(mod~$6$)} and $s\equiv1$~{\em(mod~$6$)},
then the Ap\'ery numbers $a_n(r,s)$ obey the following congruences
modulo $9$:
\begin{enumerate} 
\item [(i)]
$a_n(r,s)\equiv 1$~{\em(mod~$9$)} if, and only if,
the $3$-adic expansion of $n$ contains $0$'s and $2$'s only;
\item [(ii)]
$a_n(r,s)\equiv 3$~{\em(mod~$9$)} if, and only if,
the $3$-adic expansion of $n$ contains exactly one $1$,
and otherwise only $0$'s and $2$'s;
\item [(iii)]
$a_n(r,s)\equiv 6$~{\em(mod~$9$)} if, and only if,
the $3$-adic expansion of $n$ has exactly one occurrence of
the string $11$ and otherwise contains only $0$'s and $2$'s;
\item[(iv)]in the cases not covered by Items~{\em(i)}--{\em(iii),}
$a_n(r,s)$ is divisible by $9$;
in particular, 
$a_n(r,s)\not\equiv 2,4,5,7,8$~{\em(mod~$9$)} for all $n$.
\end{enumerate}
\end{corollary}

The case where
$r\equiv1$~(mod~$3$), $r\ge4$, and $s\equiv2$~(mod~$6$)
has already been treated in Corollary~\ref{thm:Ap02}.

\begin{corollary} \label{thm:Ap73}
If $r$ and $s$ are positive integers with 
$r\equiv1$~{\em(mod~$3$)}, $r\ge4$, and $s\equiv3$~{\em(mod~$6$)},
then the Ap\'ery numbers $a_n(r,s)$ obey the following congruences
modulo $9$:
\begin{enumerate} 
\item [(i)]
$a_n(r,s)\equiv 1$~{\em(mod~$9$)} if, and only if,
the $3$-adic expansion of $n$ contains $0$'s and $2$'s only;
\item [(ii)]
$a_n(r,s)\equiv 3$~{\em(mod~$9$)} if, and only if,
the $3$-adic expansion of $n$ has exactly one occurrence of
the string $11$ or of the string $21$ --- but not both ---
and otherwise contains only $0$'s and $2$'s;
\item [(iii)]in the cases not covered by Items~{\em(i)}--{\em(ii),}
$a_n(r,s)$ is divisible by $9$;
in particular, $a_n(r,s)\not\equiv 2,4,5,6,7,8$~{\em(mod~$9$)}
for all $n$.
\end{enumerate}
\end{corollary}

\begin{corollary} \label{thm:Ap74}
If $r$ and $s$ are positive integers with 
$r\equiv1$~{\em(mod~$3$)}, $r\ge4$, and $s\equiv4$~{\em(mod~$6$)},
then the Ap\'ery numbers $a_n(r,s)$ obey the following congruences
modulo $9$:
\begin{enumerate} 
\item [(i)]
$a_n(r,s)\equiv 1$~{\em(mod~$9$)} if, and only if,
the $3$-adic expansion of $n$ contains an even number of 
digits~$1$, and otherwise only $0$'s and $2$'s;
\item [(ii)]
$a_n(r,s)\equiv 8$~{\em(mod~$9$)} if, and only if,
the $3$-adic expansion of $n$ contains an odd number of 
digits~$1$, and otherwise only $0$'s and $2$'s;
\item [(iii)]in the cases not covered by Items~{\em(i)}--{\em(ii),}
$a_n(r,s)$ is divisible by $9$;
in particular, 
$a_n(r,s)\not\equiv 2,3,4,5,6,7$~{\em(mod~$9$)} for all $n$.
\end{enumerate}
\end{corollary}

\begin{corollary} \label{thm:Ap75}
If $r$ and $s$ are positive integers with 
$r\equiv1$~{\em(mod~$3$)}, $r\ge4$, and $s\equiv5$~{\em(mod~$6$)}, 
then the Ap\'ery numbers $a_n(r,s)$ obey the following congruences
modulo $9$:
\begin{enumerate} 
\item [(i)]
$a_n(r,s)\equiv 1$~{\em(mod~$9$)} if, and only if,
the $3$-adic expansion of $n$ contains $0$'s and $2$'s only;
\item [(ii)]
$a_n(r,s)\equiv 3$~{\em(mod~$9$)} if, and only if,
the $3$-adic expansion of $n$ has exactly one occurrence of
the string $21$ and otherwise contains only $0$'s and $2$'s;
\item [(iii)]
$a_n(r,s)\equiv 6$~{\em(mod~$9$)} if, and only if,
the $3$-adic expansion of $n$ has exactly one occurrence of
the string $01$ {\em(}including an occurrence of a $1$ at the
beginning{\em)} and otherwise contains only $0$'s and $2$'s;
\item[(iv)]in the cases not covered by Items~{\em(i)}--{\em(iii),}
$a_n(r,s)$ is divisible by $9$;
in particular, 
$a_n(r,s)\not\equiv 2,4,5,7,8$~{\em(mod~$9$)} for all $n$.
\end{enumerate}
\end{corollary}
\end{commkurz}

\begin{corollary} \label{thm:Ap26}
If $r$ and $s$ are positive integers with 
$r\equiv2$~{\em(mod~$3$)} and $s\equiv0$~{\em(mod~$6$)}, 
then the Ap\'ery numbers $a_n(r,s)$ obey the following congruences
modulo $9$:
\begin{enumerate} 
\item [(i)]
$a_n(r,s)\equiv 1$~{\em(mod~$9$)} if, and only if,
the $3$-adic expansion of $n$ has $2d$ digits~$1$,
$o_1$ occurrences of the string $11$, 
$o_2$ occurrences of the string $21$,
and $d+o_1-o_2\equiv0$~{\em(mod~$3$)}; 
\item [(ii)]
$a_n(r,s)\equiv 2$~{\em(mod~$9$)} if, and only if,
the $3$-adic expansion of $n$ has $2d+1$ digits~$1$,
$o_1$ occurrences of the string $11$, 
$o_2$ occurrences of the string $21$,
and $d+o_1-o_2\equiv0$~{\em(mod~$3$)}; 
\item [(iii)]
$a_n(r,s)\equiv 4$~{\em(mod~$9$)} if, and only if,
the $3$-adic expansion of $n$ has $2d$ digits~$1$,
$o_1$ occurrences of the string $11$, 
$o_2$ occurrences of the string $21$,
and $d+o_1-o_2\equiv1$~{\em(mod~$3$)}; 
\item [(iv)]
$a_n(r,s)\equiv 5$~{\em(mod~$9$)} if, and only if,
the $3$-adic expansion of $n$ has $2d+1$ digits~$1$,
$o_1$ occurrences of the string $11$, 
$o_2$ occurrences of the string $21$,
and $d+o_1-o_2\equiv2$~{\em(mod~$3$)}; 
\item [(v)]
$a_n(r,s)\equiv 7$~{\em(mod~$9$)} if, and only if,
the $3$-adic expansion of $n$ has $2d$ digits~$1$,
$o_1$ occurrences of the string $11$, 
$o_2$ occurrences of the string $21$,
and $d+o_1-o_2\equiv2$~{\em(mod~$3$)}; 
\item [(vi)]
$a_n(r,s)\equiv 8$~{\em(mod~$9$)} if, and only if,
the $3$-adic expansion of $n$ has $2d+1$ digits~$1$,
$o_1$ occurrences of the string $11$, 
$o_2$ occurrences of the string $21$,
and $d+o_1-o_2\equiv1$~{\em(mod~$3$)}; 
\item [(vii)]in the cases not covered by Items~{\em(i)}--{\em(vi),}
$a_n(r,s)$ is divisible by $9$;
in particular, 
$a_n(r,s)\not\equiv 3,6$~{\em(mod~$9$)} for all $n$.
\end{enumerate}
\end{corollary}

\begin{corollary} \label{thm:Ap27}
If $r$ and $s$ are positive integers with 
$r\equiv2$~{\em(mod~$6$)} and $s\equiv1$~{\em(mod~$6$)}, or with
$r\equiv5$~{\em(mod~$6$)}, $s\equiv1$~{\em(mod~$6$)}, and $s\ge7$,
then the Ap\'ery numbers $a_n(r,s)$ obey the following congruences
modulo $9$:
\begin{enumerate} 
\item [(i)]
$a_n(r,s)\equiv 1$~{\em(mod~$9$)} if, and only if,
the $3$-adic expansion of $n$ contains $0$'s and $2$'s only;
\item [(ii)]
$a_n(r,s)\equiv 3$~{\em(mod~$9$)} if, and only if,
the $3$-adic expansion of $n$ has exactly one occurrence of
the string $01$ {\em(}including an occurrence of a $1$ at the
beginning{\em)} and otherwise contains only $0$'s and $2$'s;
\item [(iii)]
$a_n(r,s)\equiv 6$~{\em(mod~$9$)} if, and only if,
the $3$-adic expansion of $n$ has exactly one occurrence of
the string $21$ and otherwise contains only $0$'s and $2$'s;
\item[(iv)]in the cases not covered by Items~{\em(i)}--{\em(iii),}
$a_n(r,s)$ is divisible by $9$;
in particular, 
$a_n(r,s)\not\equiv 2,4,5,7,8$~{\em(mod~$9$)} for all $n$.
\end{enumerate}
\end{corollary}

For $r=2$ and $s=1$, this corollary establishes Conjecture~65 in
\cite{KrMuAE}. 

\begin{corollary} \label{thm:Ap22}
If $r$ and $s$ are positive integers with 
$r\equiv2$~{\em(mod~$3$)} and $s\equiv2$~{\em(mod~$6$)},
then the Ap\'ery numbers $a_n(r,s)$ obey the following congruences
modulo $9$:
\begin{enumerate} 
\item [(i)]
$a_n(r,s)\equiv 1$~{\em(mod~$9$)} if, and only if,
the $3$-adic expansion of $n$ contains $6k$ digits~$1$, for some $k$,
and otherwise only $0$'s and $2$'s;
\item [(ii)]
$a_n(r,s)\equiv 2$~{\em(mod~$9$)} if, and only if,
the $3$-adic expansion of $n$ contains $6k+5$ digits~$1$, for some $k$,
and otherwise only $0$'s and $2$'s;
\item [(iii)]
$a_n(r,s)\equiv 4$~{\em(mod~$9$)} if, and only if,
the $3$-adic expansion of $n$ contains $6k+4$ digits~$1$, for some $k$,
and otherwise only $0$'s and $2$'s;
\item [(iv)]
$a_n(r,s)\equiv 5$~{\em(mod~$9$)} if, and only if,
the $3$-adic expansion of $n$ contains $6k+1$ digits~$1$, for some $k$,
and otherwise only $0$'s and $2$'s;
\item [(v)]
$a_n(r,s)\equiv 7$~{\em(mod~$9$)} if, and only if,
the $3$-adic expansion of $n$ contains $6k+2$ digits~$1$, for some $k$,
and otherwise only $0$'s and $2$'s;
\item [(vi)]
$a_n(r,s)\equiv 8$~{\em(mod~$9$)} if, and only if,
the $3$-adic expansion of $n$ contains $6k+3$ digits~$1$, for some $k$,
and otherwise only $0$'s and $2$'s;
\item[(vii)]in the cases not covered by Items~{\em(i)}--{\em(vi),}
$a_n(r,s)$ is divisible by $9$;
in particular, 
$a_n(r,s)\not\equiv 3,6$~{\em(mod~$9$)} for all $n$.
\end{enumerate}
\end{corollary}

For $r=s=2$, this corollary establishes Conjecture~66 in
\cite{KrMuAE}. 

\begin{commkurz}
\begin{corollary} \label{thm:Ap23}
If $r$ and $s$ are positive integers with 
$r\equiv2$~{\em(mod~$3$)} and $s\equiv3$~{\em(mod~$6$)},
then the Ap\'ery numbers $a_n(r,s)$ obey the following congruences
modulo $9$:
\begin{enumerate} 
\item [(i)]
$a_n(r,s)\equiv 1$~{\em(mod~$9$)} if, and only if,
the $3$-adic expansion of $n$ contains $0$'s and $2$'s only;
\item [(ii)]
$a_n(r,s)\equiv 6$~{\em(mod~$9$)} if, and only if,
the $3$-adic expansion of $n$ has exactly one occurrence of
the string $11$ or of the string $21$ --- but not both ---
and otherwise contains only $0$'s and $2$'s;
\item [(iii)]in the cases not covered by Items~{\em(i)}--{\em(ii),}
$a_n(r,s)$ is divisible by $9$;
in particular, $a_n(r,s)\not\equiv 2,3,4,5,7,8$~{\em(mod~$9$)}
for all $n$.
\end{enumerate}
\end{corollary}

\begin{corollary} \label{thm:Ap24}
If $r$ and $s$ are positive integers with 
$r\equiv2$~{\em(mod~$6$)} and $s\equiv4$~{\em(mod~$6$)}, 
then the Ap\'ery numbers $a_n(r,s)$ obey the following congruences
modulo $9$:
\begin{enumerate} 
\item [(i)]
$a_n(r,s)\equiv 1$~{\em(mod~$9$)} if, and only if,
the $3$-adic expansion of $n$ has an even number of digits~$1$,
and the difference of the number of occurrences of the
string $11$ and the number of occurrences of the string $21$
is $\equiv0$~{\em(mod~$3$)};
\item [(ii)]
$a_n(r,s)\equiv 2$~{\em(mod~$9$)} if, and only if,
the $3$-adic expansion of $n$ has an odd number of digits~$1$,
and the difference of the number of occurrences of the
string $11$ and the number of occurrences of the string $21$
is $\equiv1$~{\em(mod~$3$)};
\item [(iii)]
$a_n(r,s)\equiv 4$~{\em(mod~$9$)} if, and only if,
the $3$-adic expansion of $n$ has an even number of digits~$1$,
and the difference of the number of occurrences of the
string $11$ and the number of occurrences of the string $21$
is $\equiv2$~{\em(mod~$3$)};
\item [(iv)]
$a_n(r,s)\equiv 5$~{\em(mod~$9$)} if, and only if,
the $3$-adic expansion of $n$ has an odd number of digits~$1$,
and the difference of the number of occurrences of the
string $11$ and the number of occurrences of the string $21$
is $\equiv2$~{\em(mod~$3$)};
\item [(v)]
$a_n(r,s)\equiv 7$~{\em(mod~$9$)} if, and only if,
the $3$-adic expansion of $n$ has an even number of digits~$1$,
and the difference of the number of occurrences of the
string $11$ and the number of occurrences of the string $21$
is $\equiv1$~{\em(mod~$3$)};
\item [(vi)]
$a_n(r,s)\equiv 8$~{\em(mod~$9$)} if, and only if,
the $3$-adic expansion of $n$ has an odd number of digits~$1$,
and the difference of the number of occurrences of the
string $11$ and the number of occurrences of the string $21$
is $\equiv0$~{\em(mod~$3$)};
\item [(vii)]in the cases not covered by Items~{\em(i)}--{\em(vi),}
$a_n(r,s)$ is divisible by $9$;
in particular, 
$a_n(r,s)\not\equiv 3,6$~{\em(mod~$9$)} for all $n$.
\end{enumerate}
\end{corollary}

\begin{corollary} \label{thm:Ap25}
If $r$ and $s$ are positive integers with 
$r\equiv2$~{\em(mod~$3$)}, $s\equiv5$~{\em(mod~$6$)}, and $r,s>1$,
then the Ap\'ery numbers $a_n(r,s)$ obey the following congruences
modulo $9$:
\begin{enumerate} 
\item [(i)]
$a_n(r,s)\equiv 1$~{\em(mod~$9$)} if, and only if,
the $3$-adic expansion of $n$ contains $0$'s and $2$'s only;
\item [(ii)]
$a_n(r,s)\equiv 3$~{\em(mod~$9$)} if, and only if,
the $3$-adic expansion of $n$ has exactly one occurrence of
the string $11$ and otherwise contains only $0$'s and $2$'s;
\item [(iii)]
$a_n(r,s)\equiv 6$~{\em(mod~$9$)} if, and only if,
the $3$-adic expansion of $n$ contains exactly one $1$,
and otherwise only $0$'s and $2$'s;
\item[(iv)]in the cases not covered by Items~{\em(i)}--{\em(iii),}
$a_n(r,s)$ is divisible by $9$;
in particular, 
$a_n(r,s)\not\equiv 2,4,5,7,8$~{\em(mod~$9$)} for all $n$.
\end{enumerate}
\end{corollary}

\begin{corollary} \label{thm:Ap31}
If $r$ is a positive integer with 
$r\equiv3$~{\em(mod~$6$)},
then the Ap\'ery numbers $a_n(r,1)$ obey the following congruences
modulo $9$:
\begin{enumerate} 
\item [(i)]
$a_n(r,1)\equiv 1$~{\em(mod~$9$)} if, and only if,
the $3$-adic expansion of $n$ contains $0$'s and $2$'s only,
and the number of maximal strings of $2$'s is
$\equiv 0$~{\em(mod~$3$)};
\item [(ii)]
$a_n(r,1)\equiv 3$~{\em(mod~$9$)} if, and only if,
the $3$-adic expansion of $n$ has exactly one occurrence of $1$
but no occurrences of the strings $12$ or $21$,
or it has exactly one occurrence of the string $11$, 
and otherwise it contains only $0$'s and $2$'s;
\item [(iii)]
$a_n(r,1)\equiv 4$~{\em(mod~$9$)} if, and only if,
the $3$-adic expansion of $n$ contains $0$'s and $2$'s only,
and the number of maximal strings of $2$'s is
$\equiv 1$~{\em(mod~$3$)};
\item [(iv)]
$a_n(r,1)\equiv 6$~{\em(mod~$9$)} if, and only if,
the $3$-adic expansion of $n$ has exactly one occurrence
of the string $212$
and otherwise contains only $0$'s and $2$'s;
\item [(v)]
$a_n(r,1)\equiv 7$~{\em(mod~$9$)} if, and only if,
the $3$-adic expansion of $n$ contains $0$'s and $2$'s only,
and the number of maximal strings of $2$'s is
$\equiv 2$~{\em(mod~$3$)};
\item [(vi)]in the cases not covered by Items~{\em(i)}--{\em(v),}
$a_n(r,1)$ is divisible by $9$;
in particular, 
$a_n(r,1)\not\equiv 2,5,8$~{\em(mod~$9$)} for all $n$.
\end{enumerate}
\end{corollary}

The case where
$r\equiv3$~(mod~$6$), $s\equiv1$~(mod~$6$), and $s\ge7$
has already been treated in Corollary~\ref{thm:Ap01}.

\begin{corollary} \label{thm:Ap51}
If $r$ is a positive integer with 
$r\equiv5$~{\em(mod~$6$)},
then the Ap\'ery numbers $a_n(r,1)$ obey the following congruences
modulo $9$:
\begin{enumerate} 
\item [(i)]
$a_n(r,1)\equiv 1$~{\em(mod~$9$)} if, and only if,
the $3$-adic expansion of $n$ contains $0$'s and $2$'s only,
and the number of maximal strings of $2$'s is
$\equiv 0$~{\em(mod~$3$)};
\item [(ii)]
$a_n(r,1)\equiv 3$~{\em(mod~$9$)} if, and only if,
the $3$-adic expansion of $n$ has exactly one occurrence of $1$
but no occurrences of the strings $012$ or $210$,
and otherwise it contains only $0$'s and $2$'s;
\item [(iii)]
$a_n(r,1)\equiv 4$~{\em(mod~$9$)} if, and only if,
the $3$-adic expansion of $n$ contains $0$'s and $2$'s only,
and the number of maximal strings of $2$'s is
$\equiv 1$~{\em(mod~$3$)};
\item [(iv)]
$a_n(r,1)\equiv 6$~{\em(mod~$9$)} if, and only if,
the $3$-adic expansion of $n$ has exactly one occurrence
of the string $210$ {\em(}including an occurrence of a string $2
1$ at the
end\/{\em)} 
and otherwise contains only $0$'s and $2$'s;
\item [(v)]
$a_n(r,1)\equiv 7$~{\em(mod~$9$)} if, and only if,
the $3$-adic expansion of $n$ contains $0$'s and $2$'s only,
and the number of maximal strings of $2$'s is
$\equiv 2$~{\em(mod~$3$)};
\item [(vi)]in the cases not covered by Items~{\em(i)}--{\em(v),}
$a_n(r,1)$ is divisible by $9$;
in particular, 
$a_n(r,1)\not\equiv 2,5,8$~{\em(mod~$9$)} for all $n$.
\end{enumerate}
\end{corollary}

The case where
$r\equiv5$~(mod~$6$), $s\equiv1$~(mod~$6$), and $s\ge7$,
has already been treated in Corollary~\ref{thm:Ap27}.

The Ap\'ery numbers $a_n(1,0)$ are not interesting in our
context, since $a_n(1,0)=2^n$. This is the reason why we ignore
the value $r=1$ in the following corollary.

\begin{corollary} \label{thm:Ap70}
Let $s(n)$ denote the sum of digits in the $3$-adic expansion of $n$.
If $r$ is a positive integer with 
$r\equiv1$~{\em(mod~$6$)} and $r\ge7$,
then the Ap\'ery numbers $a_n(r,0)$ obey the following congruences
modulo $9$:
\begin{enumerate} 
\item [(i)]
$a_n(r,0)\equiv 1$~{\em(mod~$9$)} if, and only if,
the $3$-adic expansion contains
$o_1$ occurrences of the string $11$, 
$o_2$ occurrences of the string $21$,
$o_3$ occurrences of the string $12$,
$o_4$ occurrences of the string $22$,
and $s(n)+2(-o_1+o_2+o_3-o_4)\equiv0$~{\em(mod~$6$)}; 
\item [(ii)]
$a_n(r,0)\equiv 2$~{\em(mod~$9$)} if, and only if,
the $3$-adic expansion contains
$o_1$ occurrences of the string $11$, 
$o_2$ occurrences of the string $21$,
$o_3$ occurrences of the string $12$,
$o_4$ occurrences of the string $22$,
and $s(n)+2(-o_1+o_2+o_3-o_4)\equiv1$~{\em(mod~$6$)}; 
\item [(iii)]
$a_n(r,0)\equiv 4$~{\em(mod~$9$)} if, and only if,
the $3$-adic expansion contains
$o_1$ occurrences of the string $11$, 
$o_2$ occurrences of the string $21$,
$o_3$ occurrences of the string $12$,
$o_4$ occurrences of the string $22$,
and $s(n)+2(-o_1+o_2+o_3-o_4)\equiv2$~{\em(mod~$6$)}; 
\item [(iv)]
$a_n(r,0)\equiv 5$~{\em(mod~$9$)} if, and only if,
the $3$-adic expansion contains
$o_1$ occurrences of the string $11$, 
$o_2$ occurrences of the string $21$,
$o_3$ occurrences of the string $12$,
$o_4$ occurrences of the string $22$,
and $s(n)+2(-o_1+o_2+o_3-o_4)\equiv5$~{\em(mod~$6$)}; 
\item [(v)]
$a_n(r,0)\equiv 7$~{\em(mod~$9$)} if, and only if,
the $3$-adic expansion contains
$o_1$ occurrences of the string $11$, 
$o_2$ occurrences of the string $21$,
$o_3$ occurrences of the string $12$,
$o_4$ occurrences of the string $22$,
and $s(n)+2(-o_1+o_2+o_3-o_4)\equiv4$~{\em(mod~$6$)}; 
\item [(vi)]
$a_n(r,0)\equiv 8$~{\em(mod~$9$)} if, and only if,
the $3$-adic expansion contains
$o_1$ occurrences of the string $11$, 
$o_2$ occurrences of the string $21$,
$o_3$ occurrences of the string $12$,
$o_4$ occurrences of the string $22$,
and $s(n)+2(-o_1+o_2+o_3-o_4)\equiv3$~{\em(mod~$6$)}; 
\item [(vii)]in the cases not covered by Items~{\em(i)}--{\em(vi),}
$a_n(r,s)$ is divisible by $9$;
in particular, 
$a_n(r,s)\not\equiv 3,6$~{\em(mod~$9$)} for all $n$.
\end{enumerate}
\end{corollary}

\begin{corollary} \label{thm:Ap20}
If $r$ is a positive integer with 
$r\equiv2$~{\em(mod~$6$)},
then the Ap\'ery numbers $a_n(r,0)$ obey the following congruences
modulo $9$:
\begin{enumerate} 
\item [(i)]
$a_n(r,0)\equiv 1$~{\em(mod~$9$)} if, and only if,
the $3$-adic expansion of $n$ contains no $2$, 
has an even number of digits~$1$,
and the number of maximal strings of $1$'s is
$\equiv 0$~{\em(mod~$3$)};
\item [(ii)]
$a_n(r,0)\equiv 2$~{\em(mod~$9$)} if, and only if,
the $3$-adic expansion of $n$ contains no $2$, 
has an odd number of digits~$1$,
and the number of maximal strings of $1$'s is
$\equiv 1$~{\em(mod~$3$)};
\item [(iii)]
$a_n(r,0)\equiv 3$~{\em(mod~$9$)} if, and only if,
the $3$-adic expansion of $n$ contains the string $02$
{\em(}including an occurrence of a $2$ at the
beginning{\em)}, no other digit $2$,
and an odd number of digits~$1$;
\item [(iv)]
$a_n(r,0)\equiv 4$~{\em(mod~$9$)} if, and only if,
the $3$-adic expansion of $n$ contains no $2$, 
has an even number of digits~$1$,
and the number of maximal strings of $1$'s is
$\equiv 2$~{\em(mod~$3$)};
\item [(v)]
$a_n(r,0)\equiv 5$~{\em(mod~$9$)} if, and only if,
the $3$-adic expansion of $n$ contains no $2$, 
has an odd number of digits~$1$,
and the number of maximal strings of $1$'s is
$\equiv 2$~{\em(mod~$3$)};
\item [(vi)]
$a_n(r,0)\equiv 6$~{\em(mod~$9$)} if, and only if,
the $3$-adic expansion of $n$ contains the string $02$, no other digit $2$,
and an even number of digits~$1$;
\item [(vii)]
$a_n(r,0)\equiv 7$~{\em(mod~$9$)} if, and only if,
the $3$-adic expansion of $n$ contains no $2$, 
has an even number of digits~$1$,
and the number of maximal strings of $1$'s is
$\equiv 1$~{\em(mod~$3$)};
\item [(viii)]
$a_n(r,0)\equiv 8$~{\em(mod~$9$)} if, and only if,
the $3$-adic expansion of $n$ contains no $2$, 
has an odd number of digits~$1$,
and the number of maximal strings of $1$'s is
$\equiv 0$~{\em(mod~$3$)};
\item [(ix)]
$a_n(r,0)\equiv 0$~{\em(mod~$9$)} if, and only if,
the $3$-adic expansion of $n$ contains at least two digits~$2$ or the string
$12$.
\end{enumerate}
\end{corollary}

In the special case of this corollary where $r=2$, we obtain
the Ap\'ery numbers
$a_n(2,0)=\binom {2n}n$, that is, the sequence of
central binomial coefficients. The above result
was found earlier in \cite[Cor.~33]{KrMuAE} in an equivalent form.

\begin{corollary} \label{thm:Ap30}
If $r$ is a positive integer with 
$r\equiv3$~{\em(mod~$6$)},
then the Ap\'ery numbers $a_n(r,0)$ obey the following congruences
modulo $9$:
\begin{enumerate} 
\item [(i)]
$a_n(r,0)\equiv 1$~{\em(mod~$9$)} if, and only if,
the number of digits~$1$ in
the $3$-adic expansion is $\equiv0$~{\em(mod~$6$)}; 
\item [(ii)]
$a_n(r,0)\equiv 2$~{\em(mod~$9$)} if, and only if,
the number of digits~$1$ in
the $3$-adic expansion is $\equiv1$~{\em(mod~$6$)}; 
\item [(iii)]
$a_n(r,0)\equiv 4$~{\em(mod~$9$)} if, and only if,
the number of digits~$1$ in
the $3$-adic expansion is $\equiv2$~{\em(mod~$6$)}; 
\item [(iv)]
$a_n(r,0)\equiv 5$~{\em(mod~$9$)} if, and only if,
the number of digits~$1$ in
the $3$-adic expansion is $\equiv5$~{\em(mod~$6$)}; 
\item [(v)]
$a_n(r,0)\equiv 7$~{\em(mod~$9$)} if, and only if,
the number of digits~$1$ in
the $3$-adic expansion is $\equiv4$~{\em(mod~$6$)}; 
\item [(vi)]
$a_n(r,0)\equiv 8$~{\em(mod~$9$)} if, and only if,
the number of digits~$1$ in
the $3$-adic expansion is $\equiv3$~{\em(mod~$6$)}; 
\item [(vii)]in the cases not covered by Items~{\em(i)}--{\em(vi),}
$a_n(r,s)$ is divisible by $9$;
in particular, 
$a_n(r,s)\not\equiv 3,6$~{\em(mod~$9$)} for all $n$.
\end{enumerate}
\end{corollary}

\begin{corollary} \label{thm:Ap40}
If $r$ is a positive integer with 
$r\equiv4$~{\em(mod~$6$)},
then the Ap\'ery numbers $a_n(r,0)$ obey the following congruences
modulo $9$:
\begin{enumerate} 
\item [(i)]
$a_n(r,0)\equiv 1$~{\em(mod~$9$)} if, and only if,
the $3$-adic expansion of $n$ contains no $2$, and
has $s_1$ digits~$1$ and $s_2$ maximal strings of $1$'s
with $2o_2-s_1\equiv0$~{\em(mod~6)};
\item [(ii)]
$a_n(r,0)\equiv 2$~{\em(mod~$9$)} if, and only if,
the $3$-adic expansion of $n$ contains no $2$, and
has $s_1$ digits~$1$ and $s_2$ maximal strings of $1$'s
with $2o_2-s_1\equiv1$~{\em(mod~6)};
\item [(iii)]
$a_n(r,0)\equiv 3$~{\em(mod~$9$)} if, and only if,
the $3$-adic expansion of $n$ has exactly one
occurrence of the string $12$, an odd number of digits~$1$,
and otherwise contains only $0$'s;
\item [(iv)]
$a_n(r,0)\equiv 4$~{\em(mod~$9$)} if, and only if,
the $3$-adic expansion of $n$ contains no $2$, and
has $s_1$ digits~$1$ and $s_2$ maximal strings of $1$'s
with $2o_2-s_1\equiv2$~{\em(mod~6)};
\item [(v)]
$a_n(r,0)\equiv 5$~{\em(mod~$9$)} if, and only if,
the $3$-adic expansion of $n$ contains no $2$, and
has $s_1$ digits~$1$ and $s_2$ maximal strings of $1$'s
with $2o_2-s_1\equiv5$~{\em(mod~6)};
\item [(vi)]
$a_n(r,0)\equiv 6$~{\em(mod~$9$)} if, and only if,
the $3$-adic expansion of $n$ has exactly one
occurrence of the string $12$, an even number of digits~$1$,
and otherwise contains only $0$'s;
\item [(vii)]
$a_n(r,0)\equiv 7$~{\em(mod~$9$)} if, and only if,
the $3$-adic expansion of $n$ contains no $2$, and
has $s_1$ digits~$1$ and $s_2$ maximal strings of $1$'s
with $2o_2-s_1\equiv4$~{\em(mod~6)};
\item [(viii)]
$a_n(r,0)\equiv 8$~{\em(mod~$9$)} if, and only if,
the $3$-adic expansion of $n$ contains no $2$, and
has $s_1$ digits~$1$ and $s_2$ maximal strings of $1$'s
with $2o_2-s_1\equiv3$~{\em(mod~6)};
\item[(ix)]
$a_n(r,0)\equiv 0$~{\em(mod~$9$)} if, and only if,
the $3$-adic expansion of $n$ contains at least two digits~$2$ or the string
$02$.
\end{enumerate}
\end{corollary}

\begin{corollary} \label{thm:Ap50}
Let $\widetilde s(n)$ denote the number of digits~$1$
in the $3$-adic expansion of $n$ minus $2$ times the number of digits~$2$
in the same expansion.
If $r$ is a positive integer with 
$r\equiv5$~{\em(mod~$6$)},
then the Ap\'ery numbers $a_n(r,0)$ obey the following congruences
modulo $9$:
\begin{enumerate} 
\item [(i)]
$a_n(r,0)\equiv 1$~{\em(mod~$9$)} if, and only if,
the $3$-adic expansion contains
$o_1$ occurrences of the string $11$, 
$o_2$ occurrences of the string $21$,
$o_3$ occurrences of the string $12$,
$o_4$ occurrences of the string $22$,
and $\widetilde s(n)+2(o_1-o_2-o_3+o_4)\equiv0$~{\em(mod~$6$)}; 
\item [(ii)]
$a_n(r,0)\equiv 2$~{\em(mod~$9$)} if, and only if,
the $3$-adic expansion contains
$o_1$ occurrences of the string $11$, 
$o_2$ occurrences of the string $21$,
$o_3$ occurrences of the string $12$,
$o_4$ occurrences of the string $22$,
and $\widetilde s(n)+2(o_1-o_2-o_3+o_4)\equiv1$~{\em(mod~$6$)}; 
\item [(iii)]
$a_n(r,0)\equiv 4$~{\em(mod~$9$)} if, and only if,
the $3$-adic expansion contains
$o_1$ occurrences of the string $11$, 
$o_2$ occurrences of the string $21$,
$o_3$ occurrences of the string $12$,
$o_4$ occurrences of the string $22$,
and $\widetilde s(n)+2(o_1-o_2-o_3+o_4)\equiv2$~{\em(mod~$6$)}; 
\item [(iv)]
$a_n(r,0)\equiv 5$~{\em(mod~$9$)} if, and only if,
the $3$-adic expansion contains
$o_1$ occurrences of the string $11$, 
$o_2$ occurrences of the string $21$,
$o_3$ occurrences of the string $12$,
$o_4$ occurrences of the string $22$,
and $\widetilde s(n)+2(o_1-o_2-o_3+o_4)\equiv5$~{\em(mod~$6$)}; 
\item [(v)]
$a_n(r,0)\equiv 7$~{\em(mod~$9$)} if, and only if,
the $3$-adic expansion contains
$o_1$ occurrences of the string $11$, 
$o_2$ occurrences of the string $21$,
$o_3$ occurrences of the string $12$,
$o_4$ occurrences of the string $22$,
and $\widetilde s(n)+2(o_1-o_2-o_3+o_4)\equiv4$~{\em(mod~$6$)}; 
\item [(vi)]
$a_n(r,0)\equiv 8$~{\em(mod~$9$)} if, and only if,
the $3$-adic expansion contains
$o_1$ occurrences of the string $11$, 
$o_2$ occurrences of the string $21$,
$o_3$ occurrences of the string $12$,
$o_4$ occurrences of the string $22$,
and $\widetilde s(n)+2(o_1-o_2-o_3+o_4)\equiv3$~{\em(mod~$6$)}; 
\item [(vii)]in the cases not covered by Items~{\em(i)}--{\em(vi),}
$a_n(r,s)$ is divisible by $9$;
in particular, 
$a_n(r,s)\not\equiv 3,6$~{\em(mod~$9$)} for all $n$.
\end{enumerate}
\end{corollary}

\end{commkurz}

\end{document}